\documentclass[11pt,a4paper]{article}
\title{\bf Stochastic differential equations and \\
stochastic parallel translations\\
 in the Wasserstein space }
\author{Hao DING$^{1}$\footnote{Email:  dinghao16@mails.ucas.ac.cn}
	\quad Shizan FANG$^2$\footnote{Email:Shizan.Fang@u-bourgogne.fr}
	\quad Xiang-dong LI$^{3,4}$\footnote{Email: xdli@amt.ac.cn}
	\vspace{3mm}\\
	{\footnotesize $^1$National Center for Mathematics and Interdisciplinary Sciences, }\\
	{\footnotesize Chinese Academy of Sciences, 55, Zhongguancun
		East Road, Beijing, 100190, China }\\
	{\footnotesize $^2$Institut de Math\'ematiques de Bourgogne, UMR 5584 CNRS, }\\
	{\footnotesize Universit\'e de Bourgogne Franche-Comt\'e, F-21000 Dijon, France}\\
	{\footnotesize $^3$Academy of Mathematics and Systems Science,}\\
	{\footnotesize Chinese Academy of Sciences, 55, Zhongguancun
		East Road, Beijing, 100190, China }\\
	{\footnotesize $^4$School of Mathematical Sciences, }\\
	{\footnotesize University of
		Chinese Academy of Sciences, Beijing, 100049, China}
}
\usepackage{amssymb,amsmath,amsfonts,amsthm,array,bm,color}

\def\N{\mathbb{N}}
\def\R{\mathbb{R}}
\def\E{\mathbb{E}}
\def\P{\mathbb{P}}
\def\T{\mathbb{T}}
\def\TT{\bar{\mathbf{T}}}
\def\bn{\bar{\nabla}}
\def\bD{\bar{D}}
\def\D{{\mathbb D}}

\def\F{\mathcal{F}}
\def\C{\mathcal{C}}

\def\L{\mathcal L}

\def\div{\textup{div}}

\def\eps{\varepsilon}
\def\<{\langle}
\def\>{\rangle}

\let \dis=\displaystyle
\let\ra=\rightarrow

\newtheorem{theorem}{Theorem}[section]
\newtheorem{lemma}[theorem]{Lemma}       %与theorem共用一个计数器
\newtheorem{corollary}[theorem]{Corollary}
\newtheorem{proposition}[theorem]{Proposition}
\newtheorem{remark}[theorem]{Remark}

\newtheorem{definition}[theorem]{Definition}

\begin{document}

\maketitle
\makeatletter % '@' is now a normal "letter" for TeX
\renewcommand\theequation{\thesection.\arabic{equation}}
\@addtoreset{equation}{section}
\makeatother % '@' is restored as a "non-letter" character for TeX

\vspace{-8mm}
\begin{abstract} We will develop some elements in stochastic analysis in the Wasserstein space $\P_2(M)$ over a compact Riemannian manifold $M$, such as intrinsic It\^o formulae, stochastic regular curves and parallel translations along them.  We will establish the existence of parallel translations along  regular curves, or  stochastic regular curves in case of $\P_2(\T)$. Surprisingly enough, in this 
last case, the equation defining stochastic parallel translations is a SDE on a Hilbert space, instead of  a SPDE.
\end{abstract}



\section{Introduction}\label{sect1}
{\color{black}\quad During recent decades, the study of  the optimal transportation problem has obtained great successes.  In particular,  in a seminar paper  \cite{Brenier}, Y. Brenier  proved that the optimal transport map for the long standing  Monge-Kantorovich problem with the quadratic distance cost function on Euclidean spaces is given by the gradient of a convex function, which  satisfies a Monge-Amp\`ere equation. 
In  \cite{McCann}, R. McCann extended Brenier's result to  compact Riemannian manifolds. In   \cite{Caffarelli},  L. A. Caffarelli et al. obtained the $C^{2,\alpha}$  regularity estimates for the solutions of  
Brenier Monge-Amp\`ere equation. Here we will not enter the details of this topic, but only refer to two monographs by C. Villani \cite{Villani1, Villani2}. For infinite dimensional case, see \cite{FeyelUstunel, FangShao}.

\vskip 2mm
On the other hand, R. Jordan, D. Kinderlehrer and F. Otto \cite{JKO} gave a steepest descent formulation of  the Fokker-Planck equation on Euclidean space. In  \cite{Otto}, F. Otto introduced an infinite dimensional Riemannian structure on the Wasserstein space of probabilities measures with finite second moments, and interpreted the Fokker-Planck equation and the porous medium equation
as the gradient flows associated to the Boltzmann or R\'enyi type entropy functionals on the Wasserstein space $\P_{2,\infty}(M)$, which is the space
of probability measures having positive smooth density. To obtain the long time
behavior of above equations, he developed heuristically the Hessian calculus  on the Wasserstein space $\P_{2,\infty}(M)$. A striking result is that the Hessian of  the Boltzmann entropy functional is  given by the Bakry-Emery Ricci curvature  \cite{BakryEmery}. See Otto \cite{Otto}, Otto-Villani \cite{OV}, Renesse-Sturm \cite{RS2} and Lott \cite{Lott1}.
In \cite{Sturm, LottVillani, RS2}, J. Lott, C. Villani and K.T. Sturm  developed a synthetic  geometric analysis on metric measured spaces.  In  \cite{AGS}, L. Ambrosio, N. Gigli and G. Savar\'e developed the theory of generalized gradient flows on the Wasserstein space over metric measured spaces. In \cite{LiLi, LiLi2}, S. Li and the third named author of this paper introduced the Langevin deformation between the Wasserstein geodesic flow and the gradient flows on $\P_{2,\infty}(M)$ and proved some $W$-entropy-information  formulas.}

\vskip 2mm

The definition of Levi-Civita covariant derivatives on  $\P_{2,\infty}(M)$ was introduced by J. Lott 
in \cite{Lott1}.  He also derived the linear partial differential equation for parallel translations and identified  the equation of the associated Riemannian geodesics with those introduced by Benamou and Brenier \cite{BB}  for the $L^2$-(Monge-Kantorovich-Rubinstein-)Wasserstein distance from a hydrodynamic point of view. However, the issue of existence and uniqueness of parallel transports has not been studied in \cite{Lott1}. In \cite{AG}, L. Ambrosio and N. Gigli established the existence of parallel translations along regular curves induced by flows of diffeomorphisms,
but in a weak sense; the question whether the strong solutions for parallel translations exist remains open. Concerning parallel translations along geodesics in $\P_{2,\infty}(M)$, the situation seems much more complicated, see \cite{Lott2}. \textcolor{black}{In \cite{DingFang},
the first two authors} revisited these works by proving that vector fields along a regular curve are restrictions on this curve of some vector fields defined on the whole space; this allowed us to introduce parallel translations as in differential geometry. The main purpose of this work is to introduce parallel translations along stochastic regular curves, we will establish their existence and uniqueness in the case of $\P_{2,\infty}(\T)$. Finally we remark that M.K Von Renesse and K.T. Sturm rigorously constructed  in \cite{RS1} a quasi-invariant measure on $\P_2(\T)$.

\section{Framework}\label{sect2}

In this work, we consider a connected compact Riemannian manifold $M$, of the Riemannian distance $d_M$, 
together with the normalized Riemannian measure $dx$: $\int_M dx=1$. As usual, we denote by $\P_2(M)$ 
the space of probability measures on $M$, endowed with  the Wasserstein distance $W_2$ 
defined by

\begin{equation*}
W_2^2(\mu_1, \mu_2) =\inf\Bigl\{ \int_{M\times M} d_M^2(x,y)\,\pi(dx,dy),\quad \pi\in \C(\mu_1, \mu_2)\Bigr\},
\end{equation*}
where $\C(\mu_1, \mu_2)$ is the set of probability measures $\pi$ on $M\times M$, having $\mu_1, \mu_2$ as two 
marginal laws. It is well known that $\P_2(M)$ endowed with $W_2$ is a compact space. 
\vskip 2mm

For tangent spaces $\TT_\mu$ of $\P_2(M)$ at $\mu$, we adopt the definition given in \cite{AGS}, that is, 

\begin{equation}\label{eq1.1}
\TT_\mu = \overline{\bigl\{ \nabla\psi,\ \psi\in C^\infty(M) \bigr\}}^{L^2(\mu)},
\end{equation}
the closure of gradients of smooth functions in the space $L^2(\mu)$ of vector fields on $M$. A curve $\{c(t);\ t\in [0,1]\}$ in 
$\P_2(M)$ is said to be absolutely continuous if there exists $k\in L^2([0,1])$ such that 
\begin{equation*}
W_2\bigl( c(t_1), c(t_2)\bigr)\leq \int_{t_1}^{t_2}k(s)\, ds,\quad t_1<t_2.
\end{equation*}

For such a curve, Ambrosio, Gigli and Savar\'e \cite{AGS} proved that there exists a Borel vector field $Z_t$ on $M$ in
$\dis L^2([0,1]\times M)$, that is, 
\begin{equation*}
\int_0^1 \Bigl[\int_M |Z_t(x)|_{T_xM}^2\, c_t(dx)\Bigr]\ dt<+\infty,
\end{equation*}
which satisfies the continuity equation
\begin{equation}\label{eq1.2}
\frac{dc_t}{dt} +\nabla\cdot (Z_tc_t)=0.
\end{equation}

The uniqueness of solutions to \eqref{eq1.2} holds if $Z_t\in \TT_{c_t}$ for almost all $t\in [0,1]$. We say that 
$Z_t$ is the intrinsic derivative of $\{c_t\}$ and denote it by 
\begin{equation}\label{eq1.3}
\frac{d^I c_t}{dt}.
\end{equation}

For reader's convenience, we introduce now two classes of absolutely continuous curves in $\P_2(M)$. The first example
is those  $\{c_t\}$ generated by flows of maps. More precisely, $c_t=(U_t)_\#c_0$ with 
\begin{equation}\label{eq1.3.1}
\frac{dU_t}{dt}=\nabla\psi(U_t), \quad U_0(x)=x,
\end{equation}
for regular function $\psi$ on $M$. In this case, $\dis \frac{d^Ic_t}{dt}=\nabla\psi$. 

\vskip 2mm
In what follows, we use the notation $V_\psi$ when $\nabla\psi$
is seen as a vector field on $\P_2(M)$, just as  J. Lott did in \cite{Lott1} in order to clarify different roles played by $\nabla\psi$.
For a functional $F$ on $\P_2(M)$, we say that $F$ is derivable along $V_\psi$ if 
\begin{equation}\label{eq1.4}
(\bD_{V_\psi}F)(\mu)=\Bigl\{\frac{d}{dt}F((U_t)_{\#\mu})\Bigr\}_{ t=0}\quad\hbox{\rm exists}.
\end{equation}
We say that the gradient $\bn F(\mu)\in \TT_\mu$ exists if for each $\psi\in C^\infty(M)$, $\bD_{V_\psi}F$ exists and 
$(\bD_{V_\psi}F)(\mu)=\<\bn F,V_\psi\>_{\TT_\mu}$.

\vskip 2mm

The second class of examples of such $\{c_t\}$ is the class of geodesics in $\P_2(M)$. More precisely, 
for two probability measures $c_0, c_1$ on $M$ having density, according to R. McCann \cite{McCann}, there is 
a function $\phi$ in the Sobolev space $\D_1^2(M)$ such that the optimal transport map ${\cal T}$
 pushing $c_0$ to $c_1$, admits the expression
\begin{equation}\label{eq1.5}
{\cal T}(x)=\exp_x(\nabla\phi(x)).
\end{equation}

Then $\{c_t\}$ is defined by $\dis c_t=({\cal T}_t)_\#c_0$ where ${\cal T}_t(x)=\exp_x(t\nabla\phi(x))$, and 
$\dis \frac{d^Ic_t}{dt}=\nabla\phi({\cal T}_t^{-1}(x))$. 

\vskip 2mm
In order to intrinsically formulate stochastic differential equations (SDE) on $\P_2(M)$, we need to introduce 
suitable functionals. Here are usual functionals considered in literature (see for example \cite{AGS}).
\vskip 2mm

1) Potential energy functional. For any $\varphi\in C^2(M)$, we set $F_\varphi(\mu)=\int_M \varphi\ \mu(dx)$. 

2) Internal energy functional. Let $\chi: [0,+\infty[\ra ]-\infty, +\infty]$ be a proper, continuous convex function
satisfying $\chi(0)=0, \ \liminf_{s\ra 0} \frac{\chi(s)}{s^\alpha}>-\infty$ for some $\alpha>\frac{d}{d+2}$, 
where $d$ is the dimension of $M$. The internal 
energy $\F$ is defined as follows 
\begin{equation*}
\F(\mu)=\int_M \chi(\rho(x))\, dx,\quad\hbox{\rm if }\ d\mu=\rho\, dx,
\end{equation*}
and $\F(\mu)=+\infty$ otherwise. Two important examples are $\chi(s)=s\log(s)$ and
$ \dis\chi(s)=\frac{s^m}{m-1}$ for $m> 1$.

3) Interaction energy functional. Let $W: M^2 \ra ]-\infty, +\infty]$ be a l.s.c function, we define
\begin{equation*}
{\mathcal W}(\mu)=\int_{M\times M} W(x,y)\mu(dx)\mu(dy).
\end{equation*}

\vskip 2mm
The results below are not new, but for the sake of self-contained of our paper, we give a complete proof in what follows.
We collect first and second order derivatives of these functionals in the following

\begin{proposition}\label{prop1.1} Let $V_\psi$ be a constant vector field on $\P_2(M)$, we have 
\vskip 2mm

(i)\quad $\dis (\bD_{V_\psi}F_\varphi)(\mu)=\<V_\psi, \bn F_\varphi\>_{\TT_\mu}$ and
\begin{equation}\label{eq1.6}
(\bD_{V_\psi}\bD_{V_\psi}F_\varphi)(\mu)=\int_M \L_{\nabla\psi}^2\varphi\ \mu(dx).
\end{equation}

(ii) \quad For $\chi\in C^2(\R^*)$ such that $|\chi(s)| + s|\chi'(s)|+s^2|\chi''(s)|$ is bounded over $[0,1]$, we have 
 $\dis (\bD_{V_\psi}\F)(\mu)=-\int_M\bigl(\chi'(\rho)\rho-\chi(\rho)\bigr)\, \Delta\psi\, dx$
and
\begin{equation}\label{eq1.7}
(\bD_{V_\psi}\bD_{V_\psi}\F)(\mu)=\int_M \textcolor{black}{p}'(\rho)(\Delta\psi)^2\rho^2\, dx
-\int_M \textcolor{black}{p}(\rho)\<\nabla\psi, \nabla\Delta\psi>\, \rho\, dx,
\end{equation}
where $\dis \textcolor{black}{p}(s)=\chi'(s)-\frac{\chi(s)}{s}$. For $\chi(s)=s\log s$, 
$\dis (\bD_{V_\psi}\bD_{V_\psi}\F)(\mu)=-\int_M  \<\nabla\psi, \nabla\Delta\psi\>\, \rho\, dx$.

(iii)\quad Set 
\begin{equation*}
\Phi(x,\mu)=\int_M \bigl( W(x,y)+W(y,x)\bigr)\, \mu(dy).
\end{equation*}
Then $\dis (\bD_{V_\psi}{\mathcal W})(\mu)=\int_M \<\nabla\Phi(x,;\mu), \nabla\psi(x)\>\, \mu(dx)$, and 
\begin{equation*}
(\bD_{V_\psi}\bD_{V_\psi}{\cal W})(\mu)= 
\int_M \bigl((\bD_{V_\psi}\L_{\nabla\psi}\Phi)(x,\mu)+(\L^2_{\nabla\psi}\Phi)(x,\mu)\bigr)\, \mu(dx).
\end{equation*}
Where $\L_{\nabla\psi}$ denotes the Lie derivative with respect to $\nabla\psi$ on $M$.
\end{proposition}

\vskip 2mm
\begin{proof} Item (i) is well-known. For proving (ii), we consider the flow  $(U_t)$ associated to $\nabla\psi$ (see \eqref{eq1.3.1}).
Let $\mu_t=(U_t)_\#(\rho dx)=\rho_t\, dx$. It is known (see \cite{ABC}) that
\begin{equation*}
\rho_t=\rho(U_{-t})\,e^{-\int_0^t \Delta\psi(U_{-s})ds}
\quad
\hbox{\rm or}\quad
\rho_t(U_t)=\rho\,e^{-\int_0^t \Delta\psi(U_{t-s})ds}.
\end{equation*}
It follows that 
\begin{equation}\label{eq1.9}
\Bigl\{\frac{d}{dt}\rho_t(U_t)\Bigr\}_{t=0}=- (\Delta\psi)\, \rho.
\end{equation}
We have $\dis\F(\mu_t)=\int_M \chi(\rho_t)\,dx=\int_M \frac{\chi(\rho_t)}{\rho_t}\, \rho_tdx=\int_M \hat\chi(\rho_t(U_t))\, \rho dx$,
where $\hat\chi(s)=\chi(s)/s$. Therefore by  Relation \eqref{eq1.9},  we get 
\begin{equation*}
\Bigl\{\frac{d}{dt}\F(\mu_t)\Bigr\}_{t=0}=-\int_M \hat\chi'(\rho) (\Delta\psi\, \rho)\, \rho dx
=-\int_M (\chi'(\rho)\rho-\chi(\rho))\Delta\psi\, dx.
\end{equation*}
For the sake of simplicity, we denote for a moment by $\tilde\F$ the right hand of above relation. Then 
\begin{equation*}
\tilde\F(\mu_t)=-\int_M \bigl(\chi'(\rho_t)\rho_t-\chi(\rho_t)\bigr)\Delta\psi\, dx
=-\int_M \textcolor{black}{p}(\rho_t)\, \Delta\psi\, \rho_t\, dx
\end{equation*}
which is equal to 
\begin{equation*}
-\int_M \textcolor{black}{p}(\rho_t(U_t))\, \Delta\psi(U_t)\,  \rho dx.
\end{equation*}

Again using Relation \eqref{eq1.9}, we get 

\begin{equation*}
\Bigl\{\frac{d}{dt}\tilde\F(\mu_t)\Bigr\}_{t=0}=\int_M \textcolor{black}{p}'(\rho) (\Delta\psi\, \rho)\, \Delta\psi\rho dx
-\int_M \textcolor{black}{p}(\rho)\,\<\nabla\Delta\psi, \nabla\psi\>\,\rho dx,
\end{equation*}
which is nothing but the second formula in (ii). Item (iii) comes from direct calculation.
\end{proof}

\section{Stochastic differential equations on $\P_2(M)$}\label{sect3}

Let's first say a few words on SDE on a Riemannian manifold $M$. Given a family of vector fields 
$\{A_0(t,\cdot), A_1(t,\cdot), \ldots, A_N(t,\cdot)\}$ on $M$,  and $B_t^0=t$ 
and  a standard Brownian motion $t\ra (B_t^1, \ldots, B_t^N)$ on $\R^N$,
how to understand the following SDE on $M$
\begin{equation}\label{eq2.1}
dX_{t,s}=\sum_{i=0}^N A_i(t, X_{t,s})\circ dB_t^i,\quad X_{s,s}(x)=x \ ?
\end{equation}
The equality \eqref{eq2.1} formally holds  in the tangent space $T_{X_{t,s}}M$. 
 Rigorously, a stochastic process $\{X_{t,s},\ t\geq s\}$ is a solution 
to SDE \eqref{eq2.1} if for any test function $f\in C^2(M)$, 
it holds 
\begin{equation}\label{eq2.2}
\begin{split}
f(X_{t,s})=&f(x)+\sum_{i=1}^N \int_s^t (\L_{A_i(u)}f)(X_{u,s})\,dB_u^i\\
&+ \int_s^t \Bigl(\bigl( \L_{A_0(u)}f+\frac{1}{2}\sum_{i=1}^N \L^2_{A_i(u)}f\bigr)(X_{u,s})\, \Bigr)\, du,
\end{split}
\end{equation}
where $\L_{A_i(u)}$ denotes the Lie derivative on $M$ with respect to $x\ra A_i(u,x)$, see \cite{Elworthy, IW, Malliavin}.

\vskip 2mm
In what follows, we develop this concept and use three types of functionals  in the preceding section 
 to introducing SDE on $\P_2(M)$. 
 
 \vskip 2mm
 Let $\{\phi_0, \phi_1, \ldots, \phi_N\}$ be a family of functions on $[0,1]\times M$,
  continuous in $t\in [0,1]$ and smooth in $x\in M$. In this work, $\nabla$ always denotes the gradient operator on $M$.
  First we consider the following Stratanovich SDE on $M$:
  
\begin{equation}\label{eq2.3}
dX_{t,s}=\sum_{i=0}^N \nabla\phi_i(t, X_{t,s})\circ dB_t^i,\quad t\geq s, \quad X_{s,s}(x)=x.
\end{equation}
Let $d\mu=\rho\, dx$ be a probability measure on $M$, we set $\dis \mu_t(\omega)=\bigl(X_{t,0}(\omega)\bigr)_\#\mu$.
Let $\varphi\in C^2(M)$. After first using It\^o\textcolor{black}{'s} formula to $\varphi(X_{t,0})$, then integrating the two hand sides respect to $d\mu$,
we get 

\begin{equation*}
\circ d_t\,F_\varphi(\mu_t)=\sum_{i=0}^N \Bigl( \int_M \<\nabla\varphi, \nabla\phi_i(t, \cdot)\>\ \mu_t(dx)\Bigr) \circ dB_t^i
=\sum_{i=0}^N \<V_\varphi, V_{\phi_i(t,\cdot)} \>_{\TT_{\mu_t}}\circ dB_t^i.
\end{equation*}

\begin{definition}
We  say that the intrinsic It\^o stochastic differential of $\mu_t$, denoted by $\circ d_t^I\mu_t$, admits the following expression
\begin{equation}\label{eq2.4}
\circ d_t^I\mu_t=\sum_{i=0}^N V_{\phi_i(t,\cdot)}\  \circ dB_t^i.
\end{equation}
\end{definition}

Recall that $\bn F_\varphi=V_\varphi$; using notation \eqref{eq2.4}, $\circ d_tF_\varphi(\mu_t)$ can be written in the form

\begin{equation*}
\circ d_t\, F_\varphi(\mu_t)=\langle \bn F_\varphi,\ \circ d_t^I \mu_t\rangle_{\TT_{\mu_t}},
\end{equation*}
symbolically read in the inner product of $\TT_{\mu_t}$, in the same way as we did on Riemannian manifolds.
\vskip 2mm

By It\^o\textcolor{black}{'s} formula \eqref{eq2.2}, 
\begin{equation*}
d_t\varphi(X_{t,0})= \sum_{i=0}^N  (\L_{\nabla \phi_i(t)} \varphi)(X_{t,0})\, dB_t^i
+\frac{1}{2}\ \sum_{i=1}^N \bigl(\L_{\nabla\phi_i(t)}^2\varphi\bigr)(X_{t,0})\, dt.
\end{equation*}

According to Proposition \ref{prop1.1}, above relation yields 
\begin{equation*}
d_t F_\varphi (\mu_t)=\sum_{i=0}^N \<\bn F_\varphi, V_{\phi_i(t)}\>_{\TT_{\mu_t}}\, dB_t^i
+\frac{1}{2} \sum_{i=1}^n (\bD_{V_{\phi_i(t)}}^2F_\varphi)(\mu_t)\, dt.
\end{equation*}

\begin{proposition}\label{prop2.1} For any polynomial $F$ on $\P_2(M)$, 
we have
\begin{equation}\label{eq2.5}
d_t F(\mu_t) = \sum_{i=0}^N \langle \bn F, V_{\phi_i(t,\cdot)}\rangle_{\TT_{\mu_t}}\ dB_t^i
+\frac{1}{2}\sum_{i=1}^N (\bar D_{V_{\phi_i(t,\cdot)}}^2F)(\mu_t)\ dt.
\end{equation}
\end{proposition}

\vskip 2mm
\begin{proof}  For two functionals $F_1$ and $F_2$ satisfying Formula \eqref{eq2.5}, by It\^o formula, 
\begin{equation*}
 d_t(F_1F_2)(\mu_t)=d_tF_1(\mu_t)\, F_2(\mu_t)+F_1(\mu_t)\, d_tF_2(\mu_t)+ d_tF_1(\mu_t)\cdot d_tF_2(\mu_t).
 \end{equation*}
 
  Note that 
\begin{equation*}
\bar D_{V_{\phi_i(t,\cdot)}}^2 (F_1F_2)
=F_2\bar D_{V_{\phi_i(t,\cdot)}}^2F_1 +F_1 \bar D_{V_{\phi_i(t,\cdot)}}^2F_2
+2 \langle \bn F_1, V_{\phi_i(t,\cdot)}\rangle\cdot \langle \bn F_2, V_{\phi_i(t,\cdot)}\rangle,
\end{equation*}
and $\dis d_tF_1(\mu_t)\cdot d_tF_2(\mu_t)
=\sum_{i=1}^N \langle \bn F_1, V_{\phi_i(t,\cdot)}\rangle\cdot \langle \bn F_2, V_{\phi_i(t,\cdot)}\rangle\, dt$; 
so Formula \eqref{eq2.5} holds for $F_1F_2$. A polynomial $F$ on $\P_2(M)$ is a finite sum of $F_{\varphi_1}\cdots F_{\varphi_k}$,
Formula \eqref{eq2.5} remains true for $F$. We complete the proof. 
\end{proof}

\vskip 2mm
Now are going to see what happens with internal energy functional $\F$, which is not continuous.
Note that if $\dis \rho=\frac{d\mu}{dx}>0$, 
then for almost all $\omega$, $\mu_t(\omega)$ has a positive density $\rho_t$ with respect to $dx$.
\vskip 2mm

\begin{proposition}\label{prop2.2} The stochastic process $\{\rho_t, t\geq 0\}$ satisfies the following SPDE 
\begin{equation}\label{eq2.6} 
d\rho_t=-\sum_{i=0}^N \div\bigl(\rho_t\,\nabla\phi_i(t)\bigr)\, dB_t^i
+\frac{1}{2} \sum_{i=1}^N \div\bigl(\div(\rho_t\nabla\phi_i(t))\nabla\phi_i(t)\bigr)\, dt.
\end{equation}
\end{proposition}

\vskip 2mm
\begin{proof} We have
\begin{equation*}
\begin{split}
\int_M \<\nabla\varphi, \nabla\phi_i(t)\>\, \mu_t(dx)
&=\int_M \<\nabla\varphi, \rho_t\nabla\phi_i(t)\>\, dx\\
&=-\int_M \varphi\, \div\bigl(\rho_t\nabla\phi_i(t)\bigr)\, dx.
\end{split}
\end{equation*}
In the same way,

\begin{equation*}
\int_M \L_{\nabla\phi_i(t)}^2\varphi\, \mu_t(dx)=\int_M \varphi\, \div\bigl(\div(\rho_t\nabla\phi_i(t))\nabla\phi_i(t)\bigr)\ dx.
\end{equation*}
Formula \eqref{eq2.6} follows by considering $F_\varphi(\mu_t)=\int_M \varphi\, \rho_t\, dx$ for any $\varphi\in C^2(M)$, 
together with Formula \eqref{eq2.5}. 
\end{proof}

\vskip 2mm
Let $\chi$ be a $C^2$ function defined on $]0,+\infty[$. According to \eqref{eq2.6}, we have
\begin{equation*}
\begin{split}
d\chi(\rho_t)=-\sum_{i=0}^N \chi'(\rho_t)\, \div\bigl(\rho_t\nabla\phi_i(t)\bigr)\, dB_t^i
&+\frac{1}{2}\sum_{i=1}^N \chi'(\rho_t)\,\div\bigl(\div(\rho_t\nabla\phi_i(t))\nabla\phi_i(t) \bigr)\,dt\\
&+\frac{1}{2}\sum_{i=1}^N \chi''(\rho_t)\bigl(\div(\rho_t\nabla\phi_i(t)) \bigr)^2\, dt.
\end{split}
\end{equation*}

Remarking that $\dis  \div\bigl(\rho_t\nabla\phi_i(t)\bigr)=\rho_t\Delta\phi_i(t)+\<\nabla\rho_t, \nabla\phi_i(t)\>$, 
that we will use several times in the sequel, we have
\begin{equation*}
\int_M \chi'(\rho_t)\, \div\bigl(\rho_t\nabla\phi_i(t)\bigr)\, dx=\int_M \chi'(\rho_t)\, \rho_t\Delta\phi_i(t)\, dx
+\int_M \chi'(\rho_t)\<\nabla\rho_t, \nabla\phi_i(t)\>\, dx,
\end{equation*}
this last term is equal to 
\begin{equation*}
\int_M \<\nabla\bigl(\chi(\rho_t)\bigr), \nabla\phi_i(t)\>\, dx
=-\int_M \chi(\rho_t)\, \Delta\phi_i(t)\, dx.
\end{equation*}
Therefore
\begin{equation*}
\int_M \chi'(\rho_t)\, \div\bigl(\rho_t\nabla\phi_i(t)\bigr)\, dx=\int_M\bigr( \chi'(\rho_t)\rho_t - \chi(\rho_t)\bigr)\,\Delta\phi_i(t)\, dx,
\end{equation*}
which is $-(\bD_{V_{\phi_i(t)}} \F)(\mu_t)$ by Proposition \ref{prop1.1}. Now we are going to deal with drift terms. 
First of all, we compute 
\begin{equation*}
I_1=\int_M  \chi'(\rho_t)\,\div\bigl(\div(\rho_t\nabla\phi_i(t))\nabla\phi_i(t) \bigr)\,dx
=-\int_M \chi''(\rho_t)\,\<\nabla\rho_t, \nabla\phi_i(t)\>\, \div(\rho_t\nabla\phi_i(t))\, dx.
\end{equation*}

Again using $\dis  \div\bigl(\rho_t\nabla\phi_i(t)\bigr)=\rho_t\Delta\phi_i(t)+\<\nabla\rho_t, \nabla\phi_i(t)\>$, we write down
$I_1$ as 

\begin{equation*}
I_1=-\int_M  \chi''(\rho_t) \rho_t\, \<\nabla\rho_t, \nabla\phi_i(t)\>\Delta\phi_i(t)\,dx
-\int_M \chi''(\rho_t)\,\<\nabla\rho_t, \nabla\phi_i(t)\>^2\, dx.
\end{equation*}
Put $\dis I_2= \int_M \chi''(\rho_t)\, \bigl( \div(\rho_t\nabla\phi_i(t))\bigr)^2\, dx$. In the same way,
\begin{equation*}
\begin{split}
I_2=\int_M \chi''(\rho_t) \rho_t^2(\Delta\phi_i(t))^2\, dx &+ 2\int_M \chi''(\rho_t)\rho_t\<\nabla\rho_t, \nabla\phi_i(t)\>\,\Delta\phi_i(t)\, dx\\
&+\int_M \chi''(\rho_t) \<\nabla\rho_t, \nabla\phi_i(t)\>^2\, dx.
\end{split}
\end{equation*}

Combining above terms gives 
\begin{equation}\label{eq2.7}
I_1+I_2=\int_M \chi''(\rho_t) \rho_t^2(\Delta\phi_i(t))^2\, dx + \int_M \chi''(\rho_t)\rho_t\,\<\nabla\rho_t, \nabla\phi_i(t)\>\,\Delta\phi_i(t)\, dx.
\end{equation}
On the other hand, according to the expression for $\bD_{V_{\phi_i(t)}}^2\F$ in Proposition \ref{prop1.1}, we set
\begin{equation*}
I_3=\int_M \textcolor{black}{p}'(\rho_t)\,(\Delta\phi_i(t))^2\, \rho_t^2\,dx.
\end{equation*}
Replacing $\tilde\chi$ by its expression, $I_3$ becomes
\begin{equation*}
\begin{split}
I_3=\int_M  \chi''(\rho_t)\,(\Delta\phi_i(t))^2\, \rho_t^2\,dx&-\int_M \chi'(\rho_t)(\Delta\phi_i(t))^2\, \rho_t\,dx\\
&+\int_M  \chi(\rho_t)\,(\Delta\phi_i(t))^2\, dx.
\end{split}
\end{equation*}
We put 
\begin{equation*}
I_4 =\int_M \chi'(\rho_t)\, \<\nabla\phi_i(t), \nabla\Delta\phi_i(t)\>\, \rho_t\, dx
-\int_M \chi(\rho_t)\, \<\nabla\phi_i(t), \nabla\Delta\phi_i(t)\>\, dx.
\end{equation*}
In order to more clearly see the relation between $I_1+I_2$ and $I_3-I_4$, we now deal with the second term in $I_1+I_2$, 
that is, 

\begin{equation*} 
J=\int_M \chi''(\rho_t)\rho_t\,\<\nabla\rho_t, \nabla\phi_i(t)\>\,\Delta\phi_i(t)\, dx
=\int_M \<\nabla\bigl(\chi'(\rho_t)),\nabla\phi_i(t)\>\, \rho_t\Delta\phi_i(t)\, dx
\end{equation*}
 Via divergence operator $\div$, we get 
\begin{equation*}
J=-\int_M \chi'(\rho_t)\, \div\bigl(\rho_t\Delta\phi_i(t)\,\nabla\phi_i(t)\bigr)\ dx.
\end{equation*}
Note that $\dis \div\bigl(\rho_t\Delta\phi_i(t)\,\nabla\phi_i(t)\bigr)=\rho_t\, (\Delta\phi_i(t))^2+
\rho_t\, \<\nabla\Delta\phi_i(t),\nabla\phi_i(t)\>+ \<\nabla\rho_t,\nabla\phi_i(t)\>\Delta\phi_i(t)$, and
\begin{equation*}
\begin{split}
-&\int_M \chi'(\rho_t)\<\nabla\rho_t,\nabla\phi_i(t)\>\Delta\phi_i(t)\, dx
=\int_M  \chi(\rho_t)\, \div\bigl(\Delta\phi_i(t)\nabla\phi_i(t)\bigr)\,dx\\
&=\int_M \chi(\rho_t)\, \bigl((\Delta\phi_i(t))^2+\<\nabla\phi_i(t), \nabla\Delta\phi_i(t)\>\bigr)\, dx.
\end{split}
\end{equation*}
Combining all of these terms, finally we prove that $I_1+I_2=I_3-I_4$. In other words, we get the following result.

\begin{proposition}\label{prop2.3} For any internal energy functional $\F$ with $\chi\in C^2(]0,+\infty[)$, we have
\begin{equation*}
d_t\F(\mu_t)=\sum_{i=0}^N \<\bn \F, V_{\phi_i(t)}\>_{\TT_{\mu_t}}\, dB_t^i
+\frac{1}{2}\sum_{i=1}^N \bigl(\bD_{V_{\phi_i(t)}}^2\F\bigr)(\mu_t)\, dt.
\end{equation*}
\end{proposition}

\begin{proposition}\label{prop2.4}
It\^o\textcolor{black}{'s} formula \eqref{eq2.5} also holds for interaction energy functional ${\mathcal W}$.
\end{proposition}
\vskip 2mm

\begin{proof}
Applying It\^o formula to $\dis W\bigl(X_{t,0}(x), X_{t,0}(y)\bigr)$ and proceeding as above yields the result.
\end{proof}

\vskip 2mm
Now having these results in hand, we say now that the stochastic process $\{\mu_t;\ t\geq 0\}$ 
solves the following SDE on $\P_{2,\infty}(M)$, 

\begin{equation}\label{eq2.8}
\circ d_t^I\mu_t=\sum_{i=0}^N V_{\phi_i(t)}(\mu_t)\circ dB_t^i,\quad \mu_0=\rho\, dx.
\end{equation}

In what follows, we introduce quite general vector fields on $\P_2(M)$.
\begin{definition}\label{def2.2}
We say that $Z$ is a vector field on $\P_2(M)$ if there exists a Borel map $\Phi: M\times\P_2(M)\ra \R$
such that, for any $\mu\in\P_2(M)$, $x\ra \Phi(x,\mu)$ is $C^1$ and $\dis Z(\mu)=V_{\Phi(\cdot,\mu)}$. 
\end{definition}

Recall that for two probability measures $\mu, \nu\in \P_{2,\infty}(M)$, there is a unique optimal transport map 
$T_{\mu,\nu}: M\ra M$, which pushes $\mu$ to $\nu$ and has the expression  (see \cite{McCann,GMcCann,Villani1}):
\begin{equation*}
T_{\mu,\nu}(x)=\exp_x\bigl(\nabla\xi(x)\bigr).
\end{equation*}
Let $\dis \xi_x^{\mu,\nu}(t)=\exp_x\bigl(t\nabla\xi(x)\bigr)$ and let $//_t^{\xi_x^{\mu,\nu}}$ be the parallel translation
along $\{\xi_x^{\mu,\nu}(t);\ t\in [0,1]\}$.

\begin{definition}\label{def2.3} We say that $Z$ is Lipschitzian if there exists a constant $\kappa>0$ such that
\begin{equation}\label{eq2.9}
\int_M \Bigl|//_1^{\xi_x^{\mu,\nu}}\nabla\Phi(x,\mu)-\nabla\Phi\bigl(T_{\mu,\nu}(x),\nu\bigr)\Bigr|^2\, \mu(dx)
\leq \kappa^2\, W_2^2(\mu,\nu),
\end{equation}
for any couple of probability measures $(\mu,\nu)\in \P_{2,\infty}(M)\times  \P_{2,\infty}(M)$.
\end{definition}

It was proved in \cite{DingFang} that, under above Lipschitzian condition, the ODE on $\P_{2,\infty}(M)$
\begin{equation}\label{eq2.10}
\frac{d^I \nu_t}{dt}=Z(\nu_t);\quad \nu_0=\rho\, dx \ \hbox{\rm given}
\end{equation}
admits a unique solution. Condition \eqref{eq2.9} is satisfied \cite{DingFang} if (i) $x\ra \nabla^2\Phi(x,\mu) $ exists such that 
\begin{equation}\label{eq2.11}
\Lambda_1=\sup_{\mu\in \P_2(M)}||\nabla^2\Phi(\cdot, \mu)||_\infty <+\infty,
\end{equation}
and (ii) there exists a constant $\Lambda_2$ such that
\begin{equation}\label{eq2.12}
|\nabla\Phi(x,\mu)-\textcolor{black}{\nabla} \Phi(x,\nu)|\leq \Lambda_2\, W_2(\mu, \nu),\quad \hbox{\rm for any }\ x\in M.
\end{equation}

\begin{remark} The function $\Phi$ involved in the interaction energy functional ${\mathcal W}$
satisfies Condition \eqref{eq2.11} and \eqref{eq2.12} if $W\in C^2(M\times M)$. This functional ${\mathcal W}$ plays 
an important role in the work \cite{Liming1}. A definition of absolutely continuous vector field using parallel translations
was given in Chapter 3  in \cite{Gigli}.
\end{remark}

In the sequel, for simplicity, we suppose that $\phi_0, \ldots \phi_N$ are time-independent. 
The following result was established by F. Y. Wang in \cite{Wang} in a quite general setting. 
For the own interest, we give a construction of solutions.

\begin{proposition}\label{prop2.8} Suppose that $\Phi$ satisfies Conditions \eqref{eq2.11} and \eqref{eq2.12}, then
there is a unique solution $(X_t, \mu_t)$ to the following Mckean-Vlasov SDE:
\begin{equation}\label{eq2.13}
dX_t=\sum_{i=0}^N \nabla\phi_i(X_t)\circ dB_t^i + \nabla\Phi(X_t, \mu_t)\,dt,\ \mu_t=(X_t)_\#\mu_0,
\end{equation}
\end{proposition}
\vskip 2mm

\begin{proof}  We will construct a solution to \eqref{eq2.13}. Let $(U_t)_{t\geq 0}$ be the stochastic flow associated to the folllowing SDE
\begin{equation*}
dU_t = \sum_{i=0}^N \nabla\phi_i(U_t)\circ dB_t^i.
\end{equation*}
Define the stochastic measure-dependent vector fields $V_t(\omega, x, \mu)$ on $M$ by 
\begin{equation*}
V_t(\omega, x, \mu)=\bigl(U_t^{-1}(\omega, \cdot) \bigr)_*\nabla\Phi(x, (U_t)_\#\mu)
=(U_t^{-1})'(\omega, U_t(x))\nabla\Phi\bigl(U_t(x), (U_t)_\#\mu\bigr),
\end{equation*}
where the prime denotes the differential with respect to $x$. Since the manifold $M$ is compact, we have, for $\omega$ given, 
\begin{equation*}
|V_t(\omega, x, \mu)-V_t(\omega, x, \nu)|
\leq ||(U_t^{-1})'||_\infty\ |\nabla\Phi\bigl(U_t(x), (U_t)_\#\mu\bigr)-\nabla\Phi\bigl(U_t(x), (U_t)_\#\nu\bigr)|.
\end{equation*}
Note that 
\begin{equation*}
W_2\bigl( (U_t)_\#\mu, (U_t)_\#\nu\bigr)\leq ||U_t'||_\infty\,, W_2(\mu, \nu);
\end{equation*}

then under Condition \eqref{eq2.12},  $\mu\ra V_t(\omega, x, \mu)$ is Lipschitzian uniformly in $(t,x)\in [0,1]\times M$; 
 moreover  by Condition \eqref{eq2.11}, 
$x\ra V_t(\omega, x, \mu)$ is Lipschitzian uniformly in $(t,\mu)\in [0,1]\times \P_2(M)$.
So there is a unique solution $(Y_t, \nu_t)$ to the following Mckean-Vlasov ODE on $M$ 
\begin{equation*}
\frac{d}{dt}Y_t=V_t(Y_t, \nu_t),\quad \nu_t=(Y_t)_\#\mu_0.
\end{equation*}
Let $\tilde X_t=U_t(Y_t)$. By It\^o-Wentzell formula,
\begin{equation*}
d\tilde X_t=\sum_{i=0}^N \nabla\phi_i(U_t(Y_t))\circ dB_t^i + U_t'(Y_t)\, V_t(Y_t, \nu_t), 
\end{equation*}
the last term in above equality is 
$\dis \nabla\Phi\bigl( \tilde X_t, (U_t)_\#\nu_t\bigr)$.
 Note that $(\tilde X_t)_\#\mu_0=(U_t)_\#\,(Y_t)_\#\mu_0=(U_t)_\#\nu_t$; 
 therefore $\big(\tilde X_t, (U_t)_\#\nu_t\bigr)$ is a solution to 
  Mckean-Vlasov SDE \eqref{eq2.13} on $M$.
\end{proof}

\begin{proposition}\label{prop2.9}  The stochastic process $\{\mu_t;\ t\in [0,1]\}$ obtained in Proposition \ref{prop2.8} solves 
the following SDE on $\P_2(M)$:

\begin{equation}\label{eq2.14}
\circ d_t^I\mu_t=\sum_{i=0}^N V_{\phi_i}(\mu_t)\  \circ dB_t^i+ Z(\mu_t)\,dt,\quad \mu_{|_{t=0}}=\mu_0,
\end{equation}
where $Z(\mu)=V_{\Phi(\cdot,\mu)}$, see Definition \ref{def2.2}.
\end{proposition}

\section{Regular curves and parallel translations on $\P_{2,\infty}(M)$}\label{sect3}

In this section, we only consider the space $\P_{2,\infty}$. According to J. Lott \cite{Lott1}, the Levi-Civita covariant derivative
defined by usual formula 

\begin{equation*}
\begin{split}
2\langle \bn_{V_{\psi_1}}V_{\psi_2}, V_{\psi_3}\rangle_{\TT_\mu}
&= \bD_{V_{\psi_1}}\langle V_{\psi_2}, V_{\psi_3}\rangle_{\TT_\mu} + \bD_{V_{\psi_2}}\langle V_{\psi_3}, V_{\psi_1}\rangle_{\TT_\mu} 
- \bD_{V_{\psi_3}}\langle V_{\psi_1}, V_{\psi_2}\rangle_{\TT_\mu} \\
&+ \langle V_{\psi_3}, [V_{\psi_1},V_{\psi_2}]\rangle_{\TT_\mu}
- \langle V_{\psi_2}, [V_{\psi_1},V_{\psi_3}]\rangle_{\TT_\mu} - \langle V_{\psi_1}, [V_{\psi_2},V_{\psi_3}]\rangle_{\TT_\mu},
\end{split}
\end{equation*}
admits the following expression 

\begin{equation}\label{eq3.1}
\<\bn_{V_{\psi_1}}V_{\psi_2}, V_{\psi_3}\>_{\TT_\mu}=\int_M\< \nabla^2\psi_2, \nabla\psi_1\otimes\nabla\psi_3\>\, \mu(dx).
\end{equation}
Also as usual, above bracket $[V_{\psi_1}, V_{\psi_2}]$ denotes the Lie bracket of two vector fields on $\P_{2,\infty}(M)$. 
Note that  $[V_{\psi_1}, V_{\psi_2}]$ as well as $\bn_{V_{\psi_1}}V_{\psi_2}$ are not constant vector fields. Let 
\begin{equation*}
\Pi_\mu: L^2(M,TM;\mu)\ra \TT_\mu
\end{equation*}
be the orthogonal projection; then
\begin{equation}\label{eq3.2}
\bn_{V_{\psi_1}}V_{\psi_2}(\mu)=\Pi_\mu\bigl(\nabla_{\nabla\psi_1}\nabla\psi_2\bigr).
\end{equation}

For $\mu\in \P_{2,\infty(M)}$ with $d\mu=\rho\, dx,\ \rho>0$, we denote by $\Delta_\mu$ the Witten Laplacian:
\begin{equation*}
\Delta_\mu=\Delta +\<\nabla\log\rho,\ \nabla\cdot\>,
\end{equation*}
and $\div_\mu$ the divergence operator defined by $\int_M \<\nabla\varphi, Z\>\,\mu(dx)=-\int_M\varphi\, \div_\mu(Z)\, \mu(dx)$
for any $\varphi\in C^\infty(M)$. We have the relation 
$\div_\mu(Z)=\div(Z)+\<\nabla\log\rho, Z\>$. By Hodge decomposition \cite{Li}, there is a  function $f$ and a vector field $Y$
of $\div_\mu(Y)=0$ such that $Z=\nabla f + Y$. Therefore $\div_\mu(Z)=\Delta_\mu f$ 
and $f=\Delta_\mu^{-1}\bigl(\div_\mu(Z)\bigr)$, or 

\begin{equation}\label{eq3.3}
\Pi_\mu(Z)=\nabla \Delta_\mu^{-1}\bigl(\div_\mu(Z)\bigr).
\end{equation}
Formula \eqref{eq3.3} first holds for smooth $Z$, then extends to the space $L^2(M,TM; \mu)$. 
\vskip 2mm

Now let $\{c_t\}$ be a curve in $\P_{2,\infty}(M)$ defined by a flow of diffeomorphisms $X_{t,s}$ associate to ODE:
\begin{equation}\label{ODE}
dX_{t,s}=\nabla\phi_t(X_{t,s})\, dt, \quad t\geq s, \quad X_s(x)=x,
\end{equation}
with $c_t=(X_{t,0})_\#(\rho\,dx)$.

\vskip 2mm
The following derivative formula of $\Pi_{c_{t}}$ was obtained  in chapter 5 of  \cite{Gigli}.
For reader's convenience, we include a proof here.

\begin{theorem} For a smooth vector $Z$ on $M$, $t\ra \P_{c_t}(Z)$ is absolutely continuous and 
\begin{equation}\label{eq3.4}
\frac{d}{dt}\Pi_{c_t}(Z)=-\Pi_{c_t}\Bigl(\Delta_{c_t}\phi_t\, \Pi_{c_t}^\perp(Z) \Bigr),
\end{equation}
where $\Pi_\mu^\perp=I-\Pi_\mu$.
\end{theorem}
\vskip 2mm
\begin{proof}
Denote for a moment $\hat\rho_t$ the density of $c_t$ with respect to $c_0$: $\dis \hat\rho_t=\frac{dc_t}{dc_0}$, then
$\hat\rho_t$ satisfies the following equation
\begin{equation}\label{eq3.5}
\frac{d\hat\rho_t}{dt}=-\div_{c_t}(\nabla\phi_t)\, \hat\rho_t=-(\Delta_{c_t}\phi_t)\, \hat\rho_t.
\end{equation}
Let $f\in C^\infty(M)$, we have the relation 
\begin{equation*}
\int_M\<\nabla f, Z\>\, c_t(dx)=\int_M \<\nabla f, \Pi_{c_t}(Z)\>\, c_t(dx).
\end{equation*}
Using the density $\hat\rho_t$ in the two hand sides, above equality becomes 
\begin{equation*}
\int_M\<\nabla f, Z\>\, \hat\rho_t\ c_0(dx)=\int_M \<\nabla f, \Pi_{c_t}(Z)\>\, \hat\rho_t\ c_0(dx).
\end{equation*}

Taking the derivative on the two hand sides with respect to $t$, and using \eqref{eq3.5}, we get
\begin{equation*}
\begin{split}
-\int_M \<\nabla f, Z\>(\Delta_{c_t}\phi_t)\, \hat\rho_t\ c_0(dx)
=&\int_M \<\nabla f, \frac{d}{dt}\Pi_{c_t}(Z)\>\, \hat\rho_t\ c_0(dx)\\
&-\int_M \<\nabla f, \Pi_{c_t}(Z)\>\, \Delta_{c_t}\phi_t\, \hat\rho_t\ c_0(dx),
\end{split}
\end{equation*}
which implies the equality

\begin{equation*}
\int_M \<\nabla f, \frac{d}{dt}\Pi_{c_t}(Z)\>\, \hat\rho_t\ c_0(dx)
=-\int_M \<\nabla f, \Pi_{c_t}^\perp(Z)\>\, \Delta_{c_t}\phi_t\, \hat\rho_t\ c_0(dx).
\end{equation*}
By \eqref{eq3.3}, $\dis \frac{d}{dt}\Pi_{c_t}(Z)\in \TT_{c_t}$. Since $f$ is arbitrary, we get the result \eqref{eq3.4} from above 
equality.
\end{proof}

\vskip 2mm
Now let's briefly describe results obtained in the literature for parallel translations on $\P_2(M)$. Let $\{Y_t;\ t\in [0,1]\}$ 
be a family of  vector fields along $\{c_t;\ t\in [0,1]\}$, that is, $\dis Y_t\in \TT_{c_t}$. Suppose there are smooth functions 
$(t,x)\ra \Phi_t(x)$ and $(t,x)\ra \Psi_t(x)$ such that 
\begin{equation*}
\frac{d^Ic_t}{dt}=V_{\Phi_t},\quad Y_t=V_{\Psi_t}.
\end{equation*}

J. Lott obtained in \cite{Lott1} that if $\{Y_t;\ t\in [0,1]\}$ is parallel along $\{c_t;\ t\in [0,1]\}$, then 
$\{\nabla\Psi_t;\ t\in [0,1]\}$ is a solution to the following linear PDE

\begin{equation}\label{eq3.6}
\frac{d}{dt}\nabla\Psi_t+\Pi_{c_t}\Bigl(\nabla_{\nabla\Phi_t}\nabla\Psi_t\Bigr)=0.
\end{equation}

Existence \textcolor{black}{of weak solutions} to \eqref{eq3.6} were established by L. Ambrosio and N. Gigli in \cite{AG}. By 
mimicing the section 5 of the paper \cite{AG}, \textcolor{black}{the first two authors} obtained in \cite{DingFang} the followng result

\begin{theorem}\label{th3.2} For any $\nabla\Psi_0\in L^2(c_0)$, there is a unique weak solution $\{\nabla\Psi_t, t\in [0,1]\}$ in the sense that 
$V_{\Psi_t}\in \TT_{c_t}$ and
\begin{equation}\label{eq3.7}
\Pi_{c_t} \Bigl( \lim_{\eps\downarrow 0}\frac{\tau_{\eps}^{-1}\,\nabla\Psi_{t+\eps}(X_{t+\eps,t})-\nabla\Psi_t}{\eps} \Bigr)=0
\end{equation}
holds in $L^2(c_t)$ for almost all $t\in [0,1]$, where $\tau_\eps$ is the parallel translation along 
$\{s\ra X_{t+s,t}, s\in [0,\eps]\}$, that is equivalent to say that $t\ra \nabla\Psi_t$ is absolutely 
continuous and  
\begin{equation}\label{eq3.8}
\frac{d}{dt}\int_M \langle \nabla f, \nabla\Psi_t\rangle\, c_t(dx)=\int_M \langle \nabla^2f, \nabla\phi_t\otimes \nabla\Psi_t\rangle\,c_t(dx),
\quad f\in C^\infty(M).
\end{equation}
\end{theorem}

\begin{proposition}\label{prop3.3} Let $\{\nabla\Psi_t;\ t\in [0,1]\}$ be the parallel translation along $\{c_t;\ t\in [0,1]\}$ 
in Theorem \ref{th3.2}, then for any smooth vector field $Z$ on $M$, 
\begin{equation}\label{eq3.9}
\begin{split}
\frac{d}{dt}\int_M \<Z,\nabla\Psi_t\>\, c_t(dx)
=&-\int_M\<(\Delta_{c_t}\phi_t)\, \Pi_{c_t}^\perp(Z),\ \nabla\Psi_t\>\ c_t(dx)\\
&+\int_M \<\nabla_{\nabla\phi_t}\bigl(\Pi_{c_t}(Z) \bigr),\ \nabla\Psi_t\>\ c_t(dx).
\end{split}
\end{equation}
\end{proposition}

\vskip 2mm
\begin{proof}
Let $\dis I_t=\int_M \<\Pi_{c_t}(Z),\nabla\Psi_t\>\, c_t(dx)$. For any $\eps>0$, we have 
\begin{equation*}
I_{t+\eps}=\int_M \<\Pi_{c_{t+\eps}}(Z),\nabla\Psi_{t+\eps}\>\, c_{t+\eps}(dx)
=\int_M \<\tau_\eps^{-1}\Pi_{c_{t+\eps}}(Z),\tau_\eps^{-1}\nabla\Psi_{t+\eps}\>(X_{t+\eps,t})\, c_t(dx).
\end{equation*}
Then
\begin{equation*}
\begin{split}
I_{t+\eps}-I_t&=\int_M \<\tau_\eps^{-1}\Pi_{c_{t+\eps}}(Z)(X_{t+\eps,t})-\Pi_{c_t}(Z)(x),
\ \tau_\eps^{-1}\nabla\Psi_{t+\eps}(X_{t+\eps,t})\>\ c_t(dx)\\
&+\int_M \< \Pi_{c_t}(Z), \tau_\eps^{-1}\nabla\Psi_{t+\eps}(X_{t+\eps,t})-\nabla\Psi_t(x)\>\ c_t(dx)
=J_\eps^1+J_\eps^2\quad \hbox{\rm respectively}.
\end{split}
\end{equation*}
As $\eps\ra 0$, $\dis \tau_\eps^{-1}\nabla\Psi_{t+\eps}(X_{t+\eps,t})$  converges to $\nabla\Psi_t(x)$ and 
$J_\eps^2/\eps$ converges to $0$ by \eqref{eq3.7}. For $J_\eps^1$, note that 
\begin{equation*}
\begin{split}
\frac{1}{\eps}\Bigl(\tau_\eps^{-1}\Pi_{c_{t+\eps}}(Z)(X_{t+\eps,t})-\Pi_{c_t}(Z)(x) \Bigr)
=&\frac{1}{\eps} \Bigl(\tau_\eps^{-1}\Pi_{c_{t+\eps}}(Z)(X_{t+\eps,t})
-\tau_\eps^{-1}\Pi_{c_{t}}(Z)(X_{t+\eps,t})\Bigr)\\
&+\frac{1}{\eps}\Bigl(\tau_\eps^{-1}\Pi_{c_{t}}(Z)(X_{t+\eps,t})-\Pi_{c_t}(Z)(x)\Bigr)
\end{split}
\end{equation*}

As $\eps\ra 0$, the last term in above equality converges to $\dis \nabla_{\nabla\phi_t}\Pi_{c_t}(Z)$, while 
the first term on the right hand gives $\dis \frac{d}{dt}\Pi_{c_t}(Z)$. 
Finally using \eqref{eq3.4}, we get the result \eqref{eq3.9}.
\end{proof}

\vskip 2mm
Now we are going to establish the existence of strong solution to Equation \eqref{eq3.6} in the case of $\P_{2,\infty}(\T)$, 
the base manifold $M$ being the torus $\T$.

\vskip 2mm
 In this case, we can explicit the orthogonal projection $\Pi_{c_t}$. A function 
$v$ on $\T$ is the derivative of a function $\phi$ if and only if $\int_\T v(x)\, dx=0$. In order to make explicit dependence on $x$,
the derivative of $\phi$ on $\T$ is denoted by $\partial_x\phi$. Let $\mu\in \P_{2,\infty}(\T)$ 
with $\dis\rho=\frac{d\mu}{dx}>0$. Let $\partial_x\phi=\Pi_\mu(v)$; then for any function $f$,
\begin{equation*}
\int_\T \partial_xf\, v(x)\rho(x)\, dx=\int_\T \partial_xf\, \partial_x\phi\, \rho(x)\, dx.
\end{equation*}

This implies that $\partial_x(v\rho)=\partial_x(\partial_x\phi\ \rho)$, so that for a constant $K$, 

\begin{equation*}
v\rho=\partial_x\phi\, \rho+K\quad\hbox{\rm or}\quad v=\partial_x\phi+\frac{K}{\rho}.
\end{equation*}

Integrating the two hand sides over $\T$ yields $\dis K=\frac{\int_\T v(x)dx}{\int_\T\frac{dx}{\rho}}$. Therefore 

\begin{equation*}
\Pi_\mu(v)=v-\frac{\int_\T v(x)dx}{\int_\T\frac{dx}{\rho}}\cdot \frac{1}{\rho}.
\end{equation*}

For simplifying notations, we put
\begin{equation}\label{eq3.10}
\hat \rho=\frac{1}{\Bigl(\int_\T\frac{dx}{\rho}\Bigr)\, \rho}.
\end{equation}
It is obvious that $\int_\T \hat\rho\, dx=1$. In the sequel, we use notation $\Pi_\rho$ instead of $\Pi_\mu$. Then 

\begin{equation}\label{eq3.11}
\Pi_\rho(v)=v-\Bigl(\int_\T v(x)dx\Bigr)\ \hat\rho.
\end{equation}

Let $\phi_t\in C^\infty(\T)$ and $(X_t)$ be the flow associated to 
\begin{equation*}
\frac{dX_t}{dt}=\partial_x\phi_t(X_t).
\end{equation*}
Let $c_t=(X_t)_\#\mu$ with $d\mu=\rho\, dx$. Set $\dis \rho_t =\frac{dc_t}{dx}$ the density of $c_t$ with respect to $dx$.
Let $g_t\in C^2(\T)$ such that $\int_\T g_t(x)\,dx=0$. Then $\{g_t;\ t\in [0,1]\}$ is a solution to \eqref{eq3.6} if 
\begin{equation*}
\frac{dg_t}{dt}+\Pi_{\rho_t}\Bigl(\partial_xg_t\, \partial_x\phi_t\Bigr)=0.
\end{equation*}

According to \eqref{eq3.11}, we have 
\begin{equation}\label{eq3.12}
\frac{dg_t}{dt}=-\partial_xg_t\, \partial_x\phi_t+\Bigl(\int_\T \partial_xg_t\,\partial_x\phi_t\,dx\Bigr)\, \hat\rho_t.
\end{equation}

From above equation, it is easy to see that $\dis \frac{d}{dt}\int_\T g_t(x)\,dx=0$ since $\int_\T \hat\rho_t(x)dx=1$. It follows 
that $\dis \int_\T g_t(x)\, dx=\int_\T g_0(x)\, dx=0$. In other words, $g_t$ is derivative of a function on $\T$ if the initial 
condition $g_0$ does. Put $\dis f_t=g_t(X_t)$. Then 

\begin{equation*}
\frac{df_t}{dt}=\Bigl(\int_\T \partial_x g_t\,\partial_x\phi_t\,dx\Bigr)\, \hat\rho_t(X_t).
\end{equation*}

Remark that 
\begin{equation*}
\int_\T \partial_x g_t\,\partial_x\phi_t\,dx=-\int_\T g_t\, \partial_x^2\phi_t\, dx=-\int_\T \frac{g_t\,\partial_x^2\phi_t}{\rho_t}\,\rho_t\, dx
=\int_\T g_t(X_t)\, \Bigr(\frac{\partial_x^2\phi_t}{\rho_t}\Bigr)(X_t)\, \rho\,dx.
\end{equation*}
Then $f_t$ satisfies the following equation
\begin{equation}\label{eq3.13}
\frac{df_t}{dt}=-\Bigl(\int_\T f_t\ \frac{\partial_x^2\phi_t}{\rho_t}(X_t)\, \rho\,dx\Bigr)\, \hat\rho_t(X_t).
\end{equation}
Define $\dis \Lambda(t,f) = -\Bigl(\int_\T f \ \frac{\partial_x^2\phi_t}{\rho_t}(X_t)\, \rho\,dx\Bigr)\, \hat\rho_t(X_t)$.
Then above equation can be written in the form

\begin{equation*}
\frac{df_t}{dt}=\Lambda(t, f_t).
\end{equation*}

\begin{lemma}\label{lemma3.4} We have
\begin{equation}\label{eq3.14}
||\Lambda(t, f)-\Lambda(t, g)||_{L^2(\rho\, dx)}\leq \Bigl(\sup_{t\in [0,1]}||\partial_x^2\phi_t||_\infty\Bigr)\, ||f-g||_{L^2(\rho\, dx)},\quad t\in [0,1].
\end{equation}
\end{lemma}

\vskip 2mm
\begin{proof} Note that 
\begin{equation*}
\int_\T \Bigl(\frac{\partial_x^2\phi_t}{\rho_t}\Bigr)^2(X_t)\, \rho\,dx=\int_\T \frac{(\partial_x^2\phi_t)^2}{\rho_t}\, dx
\leq ||\partial_x^2\phi_t||_\infty^2\, \int_\T\frac{dx}{\rho_t},
\end{equation*}
and $\dis \int_\T \hat\rho_t(X_t)^2\, \rho\,dx=(\int_\T\frac{dx}{\rho_t})^{-1}$; it follows that
\begin{equation*}
\int_\T \Bigl|\int_\T f\ \frac{\partial_x^2\phi_t}{\rho_t}(X_t)\, \rho\,dx\Bigr|^2\, \hat\rho_t(X_t)^2\, \rho\,dx
\leq ||\partial_x^2\phi_t||_\infty^2\, ||f||_{L^2(\rho\,dx)}^2
\end{equation*}
and global Lipschitz condition \eqref{eq3.14} holds.
\end{proof}

By classical theory of ODE on Banach spaces, for $f_0\in L^2(\rho\,dx)$ given, there is a unique solution $f_t$ to 
Equation \eqref{eq3.13}. Now having this solution $(f_t)_{t\in [0,1]}$ in hand, we set
\begin{equation*}
g_t=f_t(X_t^{-1}).
\end{equation*}
We have 
\begin{equation*}
\int_\T |g_t|^2\rho_t\,dx=\int_\T |f_t|^2\,\frac{\rho_t}{\rho}(X_t)\, \tilde\rho_t\, dx,
\end{equation*}
where $\dis \tilde\rho_t =\frac{d(X_t^{-1})_\#(\rho\,dx)}{dx}$. It is known that 
$\dis\frac{\rho_t}{\rho}(X_t)\, \tilde\rho_t=\rho$. Hence 
\begin{equation}\label{eq3.15}
\int_\T |g_t|^2\rho_t\,dx=\int_\T |f_t|^2\,\rho\, dx.
\end{equation}
Now by a quite standard computation, we prove that 
\begin{equation*}
\frac{dg_t}{dt}=-\Bigl(\int_\T g_t\partial_x^2\phi_t\, dx\Bigr)\hat\rho_t-\partial_xg_t\, \partial_x\phi_t,
\end{equation*}
$g_t$ is a solution to Equation \eqref{eq3.12}, therefore $\int_\T g_t\, dx=0$.

\begin{proposition}\label{prop3.5} The solution $\{g_t;\ t\in [0,1\}]$ to Equation \eqref{eq3.12} 
preserves norms, that is
\begin{equation}\label{eq3.16}
\int_\T |g_t|^2\, \rho_t\, dx=\int_\T |g_0|^2\, \rho\,dx,\quad t\in [0,1].
\end{equation}
\end{proposition}

\vskip 2mm
\begin{proof} By \eqref{eq3.15}, it is sufficient to check $\dis \int_\T f_t\, \frac{df_t}{dt}\, \rho\,dx=0$. But by 
\eqref{eq3.13}, we compute 
\begin{equation*}
\int_\T \hat\rho_t(X_t)\, f_t\, \rho\,dx=\int_\T \hat\rho_t(X_t)g_t(X_t)\,\rho\,dx
=\int_\T \hat\rho_t g_t\rho_t\, dx=\frac{\int_\T g_t\, dx}{\int_\T \frac{dx}{\rho_t}},
\end{equation*}
this last term is equal to $0$. 
\end{proof}
Finally we get the main result of this section.

\begin{theorem} For any $g_0\in \TT_{\rho dx}$ given, there is a unique solution $g_t\in \TT_{\rho_t dx}$
to parallel translation equation \eqref{eq3.12} such that $\int_\T |g_t|^2\, \rho_t\, dx=\int_\T |g_0|^2\rho\,dx$ for any 
$t\in [0,1]$.
\end{theorem}

\section{Stochastic parallel translations}\label{sect4}

Let $\{\phi_0, \phi_1, \ldots, \phi_N\}$ be a finite family of smooth function on $M$ and $\{\mu_t;\ t\in [0,1]\}$ a solution
to SDE  \eqref{eq2.8}, which comes from SDE on $M$

\begin{equation*}
dX_t=\sum_{i=0}^N\nabla\phi_i(X_t)\circ dB_t^i,\quad\hbox{\rm with }\ B_t^0=t.
\end{equation*}

According to \eqref{eq3.6}, if a stochastic process $\{\nabla\Psi_t;\ t\in [0,1]\}$ is parallel along $\{\mu_t;\ t\in [0,1]\}$, 
Equation \eqref{eq3.6} would be replaced by the following formal Stratanovich SDE

\begin{equation}\label{eq4.1}
\circ d_t(\nabla\Psi_t)=-\sum_{i=0}^N \Pi_{\mu_t}\Bigl( \nabla_{\nabla\phi_i}\nabla\Psi_t\Bigr)\circ dB_t^i.
\end{equation}
Using notation \eqref{eq3.2}, the above equation becomes
\begin{equation}\label{eq4.2}
\circ d_t V_{\Psi_t}=-\sum_{i=0}^N \bigl(\bn_{V_{\phi_i}}V_{\Psi_t}\bigr)\circ dB_t^i.
\end{equation}

\textcolor{black} {If  \eqref{eq4.2} has a smooth solution, it preserves norm. More precisely, we have the following proposition.}

\begin{proposition}\label{prop4.1}
Suppose that $\{\nabla\Psi_t;\ t\in [0,1]\}$ is  a solution to Equation \eqref{eq4.2}, then for any $t\in [0,1]$, 
\begin{equation}\label{eq4.3}
\int_M |\nabla\Psi_t(x)|^2\mu_t(dx)=\int_M |\nabla\Psi_0(x)|^2\, \mu(dx).
\end{equation}
\end{proposition}

\vskip 2mm
\begin{proof} We give a heuristic proof of \eqref{eq4.3}. Let $\dis\rho_t=\frac{d\mu_t}{d\mu}$ be the density with respect 
to the initial probability measure $\mu$, then $\{\rho_t;\ t\in [0,1]\}$ satisfies the following SPDE:

\begin{equation*}
\circ d_t\rho_t
=-\sum_{i=0}^N \Bigl(\div_{\mu_t}(\nabla\phi_i)\rho_t \Bigr)\, \circ dB_t^i.
\end{equation*}
Now 
\begin{equation*}
\circ d_t\int_M  |\nabla\Psi_t|^2\, \rho_t\, \mu(dx)=2\int_M \<\nabla\Psi_t, \circ d_t\nabla\Psi_t\>\rho_t\, \mu(dx)
+\int_M \<\nabla\Psi_t, \nabla\Psi_t\>\, \circ d_t\rho_t\ \mu(dx).
\end{equation*}
Note that 
\begin{equation*}
\int_M \<\nabla\Psi_t, \nabla\Psi_t\>\, \div_{\mu_t}(\nabla\phi_i)\rho_t\,\mu(dx)
=2\int_M \<\nabla\Psi_t, \nabla_{\nabla\phi_i}\nabla\Psi_t\>\ \mu_t(dx).
\end{equation*}
Combining these equalities yields
\begin{equation*}
\circ d_t\int_M |\nabla\Psi_t|^2\mu_t(dx)
=2\int_M\<\nabla\Psi_t, \circ d_t\nabla\Psi_t +\sum_{i=0}^N \nabla_{\nabla\phi_i}\nabla\Psi_t\circ dB_t^i\>\ \mu_t(dx)=0.
\end{equation*}
We get \eqref{eq4.3}. 
\end{proof}

The  weak form of \eqref{eq4.2}, the stochastic counterpart of \eqref{eq3.8}, would be  
\begin{equation*} 
d_t\int_M \< \nabla f, \nabla\Psi_t\>\, \mu_t(dx) 
=\sum_{i=0}^N \Bigl( \int_M\<\nabla_{\nabla\phi_i}(\nabla f), \nabla\Psi_t\>\, \mu_t(dx)\Bigr)\circ dB_t^i .
\end{equation*}
Since $\dis \nabla_{\nabla\phi_i}(\nabla f)$ is not a vector field of gradient type, the last term in above equality 
really is 
\begin{equation*}
\sum_{i=0}^N \Bigl( \int_M\<\Pi_{\mu_t}\bigl(\nabla_{\nabla\phi_i}(\nabla f)\bigr), \nabla\Psi_t\>\, \mu_t(dx)\Bigr)\circ dB_t^i .
\end{equation*}

\begin{proposition}\label{prop4.2} For any $f\in C^3(M)$, set 
\begin{equation}\label{eq4.4}
R_t^f=\sum_{i=1}^N \Pi_{\mu_t}\Bigl(\nabla_{\nabla\phi_i}\Pi_{\mu_t} \bigl(\nabla_{\nabla_{\phi_i}}(\nabla f)\bigr) \Bigr),
\end{equation}

\begin{equation}\label{eq4.5}
S_t=\sum_{i=1}^N\Pi_{\mu_t}\Bigl((\Delta_{\mu_t}\phi_i)\,\Pi_{\mu_t}^\perp\bigl(\bigl(\nabla_{\nabla_{\phi_i}}(\nabla f)\bigr)  \bigr) \Bigr).
\end{equation}
Then the stochastic counterpart of \eqref{eq3.8} has the following form

\begin{equation*}
\begin{split}
\int_M \<\nabla f, \nabla\Psi_t\>\, \mu_t(dx)=& \int_M \<\nabla f,\nabla\Psi_0\>\, \mu(dx) 
+ \sum_{i=0}^N \int_0^t \Bigl(\int_M \< \nabla_{\nabla \phi_i}(\nabla f),\ \nabla\Psi_s\>\, \mu_s(dx) \Bigr)\, dB_s^i\\
&+  \frac{1}{2}\, \int_0^t \Bigl(\int_M \<R_s^f-S_s^f,\ \nabla\Psi_s\> \, \mu_s(dx)\Bigr)\,ds,
\end{split}
\end{equation*}
 or instrinsically
 
 \begin{equation*}
 \begin{split}
 \<V_f, V_{\Psi_t}\>_{\TT_t}
 =&\<V_f, V_{\Psi_0}\>_{\TT_\mu}+ \sum_{i=0}^N \int_0^t \< \bn_{V_{\phi_i}}V_f,\ V_{\Psi_s}\>_{\TT_{\mu_s}}\, dB_s^i \\
 &+ \frac{1}{2}\sum_{i=1}^N \int_0^t \<\bn_{V_{\phi_i}}\bn_{V_{\phi_i}}V_f, \ V_{\Psi_s}\>_{\TT_{\mu_s}}\ ds
 -\frac{1}{2}\int_0^t \<S_s^f, V_{\Psi_s}\>_{\TT_{\mu_s}}\ ds.
 \end{split}
 \end{equation*}

\end{proposition}

\vskip 2mm
\begin{proof}
We only remark that the right hand of \eqref{eq3.9} is the sum of two terms, the first term involving variations of orthogonal 
projection along the time $t$, while the second one provides the term $R_t^f$. 
\end{proof}

\vskip 2mm
Now we are going to see what happens in the case of $\P_{2,\infty}(\textcolor{black}{\T})$. For simplicity, 
we consider SDE on $\T$,
\begin{equation}\label{eq4.6}
dX_t =\partial_x\phi_t(X_t)\circ dB_t,
\end{equation}
where $B_t$ is a one-dimensional Brownian motion.
\vskip 2mm
Let $d\mu=\rho\, dx$ and $\mu_t=(X_t)_\#\mu$; set $\dis \rho_t = \frac{d\mu_t}{dx}$ the density with respect to $dx$.
Suppose that $\{\partial_x\Psi_t;\ t\in [0,1]\}$ is a solution to the equation of parallel translations:
\begin{equation}\label{eq4.7}
d_t\partial_x\Psi_t=-\Pi_{\rho_t}\bigl( \partial_x^2\Psi_t\ \partial_x\phi_t)\, dB_t
+ \Big(\frac{1}{2}R_t^{\Psi_t}+\frac{1}{2}S_t^{\Psi_t}\Bigr)\, dt.
\end{equation}
We first explicit $R_t$ and $S_t$ in this special case. Using the expression of $\Pi_\rho$ (see \eqref{eq3.11}), we have

\begin{equation*}
\begin{split}
R_t^{\Psi_t}=&\partial_x\bigl(\partial_x^2\Psi_t\, \partial_x\phi_t\bigr) \partial_x\phi_t
-\Bigl(\int_\T \partial_x^2\Psi_t\,\partial_x\phi_t\, dx\Bigr)\,\partial_x\hat\rho_t\, \partial_x\phi_t\\
&-\Bigl(\int_\T \partial_x\bigl(\partial_x^2\Psi_t\,\partial_x\phi_t \bigr)\, \partial_x\phi_t\, dx\Bigr)\textcolor{black}{\hat{\rho}_t}
+\Bigl(\int_\T \partial_x^2\Psi_t\, \partial_x\phi_t\,dx \Bigr)\Bigl(\int_\T \partial_x\hat\rho_t\, \partial_x\phi_t\ dx \Bigr)\textcolor{black}{\hat{\rho}_t}\\
&=I_1+I_2+I_3+I_4
\end{split}
\end{equation*}
respectively, and 
\begin{equation*}
\begin{split}
S_t^{\Psi_t}=& \bigl(\partial_x^2\phi_t+\partial_x\log(\rho_t)\,\partial_x\phi_t\bigr)
\Bigl(\int_\T \partial_x^2\Psi_t\, \partial_x\phi_t \ dx\Bigr)\, \hat\rho_t\\
&\textcolor{black}{-} \Bigl(\int_\T \partial_x^2\Psi_t\, \partial_x\phi_t\, dx \Bigr)
\Bigl(\int_\T \partial_x^2\phi_t\, \hat\rho_t\, dx \Bigr)\textcolor{black}{\hat{\rho}_t}\\
&\textcolor{black}{-\Bigl(\int_\T \partial_x^2\Psi_t\, \partial_x\phi_t\,dx \Bigr)} 
\Bigl(\int_\T \partial_x\log(\rho_t)\, \partial_x\phi_t\, \hat\rho_t\, dx\Bigr)\textcolor{black}{\hat{\rho}_t}\\
&=J_1+J_2+J_3+J_4
\end{split}
\end{equation*}
respectively.  We have
\begin{equation*}
I_2+J_2 =-\textcolor{black}{2\Bigl(\int_\T \partial_x^2\Psi_t\, \partial_x\phi_t\, dx \Bigr)\,\partial_x\phi_t
\partial_x\hat\rho_t},
\end{equation*}
and
\begin{equation*}
I_4+J_4 =\textcolor{black}{2\Bigl(\int_\T \partial_x^2\Psi_t\, \partial_x\phi_t\, dx \Bigr)\, \Bigl(\int_\T \partial_x\phi_t\,
\partial_x\hat\rho_t dx\Bigr)\hat{\rho}_t}.
\end{equation*}

Therefore we get the following expression for $R_t^{\Psi_t}\textcolor{black}{+}S_t^{\Psi_t}$: 
\begin{equation}\label{eq4.8}
\begin{split}
&R_t^{\Psi_t}\textcolor{black}{+}S_t^{\Psi_t}=\partial_x \bigl(\partial_x^2\Psi_t\, \partial_x\phi_t\bigr)\, \partial_x\phi_t
-\Bigl(\int_\T \partial_x\bigl(\partial_x^2\Psi_t\,\partial_x\phi_t \bigr)\, \partial_x\phi_t\, dx\Bigr)\textcolor{black}{\hat{\rho}_t}\\
& 
\textcolor{black}{+}\Bigl(\int_\T \partial_x^2\Psi_t\, \partial_x\phi_t \ dx\Bigr)\, \partial_x^2\phi_t\,\hat\rho_t
\textcolor{red}{-}\Bigl(\int_\T \partial_x^2\Psi_t\, \partial_x\phi_t \ dx\Bigr)\Bigl(\int_\T \partial_x^2\phi_t\, \hat\rho_t\, dx \Bigr)\textcolor{black}{\hat{\rho}_t}\\
&\textcolor{black}{-2\Bigl(\int_\T \partial_x^2\Psi_t\, \partial_x\phi_t\, dx \Bigr)\,\partial_x\phi_t
	\partial_x\hat\rho_t+2\Bigl(\int_\T \partial_x^2\Psi_t\, \partial_x\phi_t\, dx \Bigr)\, \Bigl(\int_\T \partial_x\phi_t\,
	\partial_x\hat\rho_t dx\Bigr)\hat{\rho}_t}.
\end{split}
\end{equation}

\vskip 2mm
Let $\dis f_t=\partial\Psi_t(X_t)$. Then by Kunita-It\^o-Wentzell formula, we get
\begin{equation*}
\begin{split}
d_tf_t=&-\Bigl(\int_\T \partial_x\Psi_t\, \partial_x^2\phi_t\, dx\Bigr)\, \hat\rho_t(X_t)\, dB_t
-\frac{1}{2} \Bigl(\int_\T \partial_x\Psi_t\, \partial_x^2\phi_t\, dx\Bigr)\, (\partial_x^2\phi_t)(X_t) \hat\rho_t(X_t)\, dt\\
&-\frac{1}{2}\Bigl(\int_\T \partial_x\Psi_t\, \partial_x\bigl( \partial_x^2\phi_t\, \partial_x\phi_t\bigr)\,dx\Bigr)\hat\rho_t(X_t)\,dt
+\frac{3}{2}\Bigl(\int_\T\partial_x\Psi_t\,\partial_x^2\phi_t\, dx \Bigr)\Bigl( \int_\T \partial_x^2\phi_t\, \hat\rho_t\, dx\Bigr)\hat\rho_t(X_t)dt.
\end{split}
\end{equation*}

As in Section \ref{sect3}, we remark that 
\begin{equation*}
\int_\T \partial_x\Psi_t\, \partial_x^2\phi_t\, dx=\int_\T f_t\times \frac{\partial_x^2\phi_t}{\rho_t}(X_t)\ \rho\,dx, 
\end{equation*}
and aslo
\begin{equation*}
\int_\T \partial_x\Psi_t\, \partial_x\bigl(\partial_x^2\phi_t\,\partial_x\phi_t\bigr)\, dx
=\int_\T f_t\times \frac{\partial_x (\partial_x^2\phi_t\, \partial_x\phi_t)}{\rho_t}(X_t)\ \rho\,dx.
\end{equation*}

We introduce two notations
\begin{equation}\label{eq4.9}
a_t=\frac{\partial_x^2\phi_t}{\rho_t}(X_t),\quad b_t=\frac{\partial_x (\partial_x^2\phi_t\, \partial_x\phi_t)}{\rho_t}(X_t).
\end{equation}
Then $\{f_t; t\in [0,1]\}$ satisfies the following equation
\begin{equation}\label{eq4.10}
\begin{split}
d_tf_t=&-\Bigl(\int_\T f_t a_t\, \rho dx\Bigr)\hat\rho_t(X_t)\, dB_t 
-\frac{1}{2}\Bigl(\int_\T f_t a_t\, \rho dx\Bigr)\bigl(\hat\rho_t\,\partial_x^2\phi_t\bigr)(X_t)\, dt\\
&-\frac{1}{2}\Bigl(\int_\T f_tb_t\, \rho dx\Bigr)\hat\rho_t(X_t)\, dt
+\frac{3}{2}\Bigl(\int_\T f_ta_t \rho dx\Bigr)\Bigl(\int_\T \partial_x^2\phi_t\hat\rho_t\, dx\Bigr)\, \hat\rho_t(X_t)\, dt.
\end{split}
\end{equation}

Let $\dis \Lambda(t,f)=-\Bigl(\int_\T f a_t\, \rho\, dx\Bigr)\, \hat\rho_t(X_t)$ and 
\begin{equation*}
\begin{split}
\Theta(t,f)=&-\frac{1}{2}\Bigl(\int_\T f_t a_t\, \rho dx\Bigr)\bigl(\hat\rho_t\,\partial_x^2\phi_t\bigr)(X_t)
-\frac{1}{2}\Bigl(\int_\T f_tb_t\, \rho dx\Bigr)\hat\rho_t(X_t)\\
&+\frac{3}{2}\Bigl(\int_\T f_ta_t \rho dx\Bigr)\Bigl(\int_\T \partial_x^2\phi_t\hat\rho_t\, dx\Bigr)\, \hat\rho_t(X_t).
\end{split}
\end{equation*}

We put $d_tf_t$ in the form
\begin{equation*}
d_tf_t=\Lambda(t, f_t)\, dB_t+\Theta(t, f_t)\, dt.
\end{equation*}

\begin{lemma}\label{lemma4.4}
 We have, for any $t\in [0,1]$ and $f\in L^2(\rho\, dx)$,
\begin{equation}\label{eq4.10-1}
||\Lambda(t, f)||_{L^2(\rho dx)}
\leq \bigl(\sup_{t\in [0,1]}||\partial_x^2\phi_t||_\infty\bigr)\, ||f||_{L^2(\rho dx)},
\end{equation}
\begin{equation}\label{eq4.10-2}
||\Theta(t, f)||_{L^2(\rho dx)}
\leq \Bigl(2\, \sup_{t\in [0,1]}||\partial_x^2\phi_t||_\infty^2\ + \sup_{t\in [0,1]}|| \partial_x (\partial_x^2\phi_t\, \partial_x\phi_t)||_\infty
\Bigr)\, ||f||_{L^2(\rho dx)}.
\end{equation}
\end{lemma}

\vskip 2mm
\begin{proof} We proceed in the same way as in the proof of Lemma \ref{lemma3.4}. 
\end{proof}

By standard Picard iteration or by SDE on Hilbert spaces, finally we get the following result.
\begin{theorem}\label{th4.1}
There is a unique solution $\{f_t; t\in [0,1]\}$ to Equation \eqref{eq4.10}. 
\end{theorem}

Define $\dis g_t=f_t(X_t^{-1})$. Contrary to ODE, we have no SDE directly expressing $X_t^{-1}$. 

\vskip 2mm

The following result will be used several times in the sequel.

\begin{lemma}\label{LEMMA}
Let $d\mu=\rho\, dx$ be a probability measure on a compact Riemannian manifold $M$ such that $\rho>0$
and $\Phi: M\ra M$ a diffeomorphism.  Set 
\begin{equation*}
\rho_\Phi =\frac{d\Phi_\#(\rho dx)}{dx}, \quad \tilde K=\frac{d (\Phi^{-1})_\#(dx)}{dx},
\end{equation*}
then   $\quad\dis \rho_\Phi(\Phi)\tilde K=\rho$.
\end{lemma}
\begin{proof}
Let $f\in C(M)$, we have 
\begin{equation*}
\int_M f \rho dx =\int_M f(\Phi^{-1}(\Phi))\, \rho dx=\int_M f(\Phi^{-1})\rho_\Phi\, dx
= \int_M f\ \rho_\Phi(\Phi)\, \tilde K\, dx,
\end{equation*}
the result follows. 
\end{proof}

\begin{proposition}\label{prop4.4} Suppose that $\dis \int_\T g_0(x) dx=0$, then 
for any $t\in [0,1]$, $\dis \int_\T g_t(x)\, dx=0$. 
\end{proposition}

\vskip 2mm
\begin{proof} Let $\dis \tilde K_t=\frac{d(X_t^{-1})_\#(dx)}{dx}$; then by Kunita \cite{Kunita} (see also \cite{FangLuoTh}), 
we have the following explicit formula: 
\begin{equation*}
\tilde K_t=\exp\Bigl(\int_0^t (\partial_x^2\phi_s)(X_s)\circ dB_s\Bigr).
\end{equation*}

Using $\tilde K_t$, $\dis \int_\T g_t(x)\, dx=\int_\T f_t\, \tilde K_t\, dx$. Remark that all of drift terms in $f_t$ came 
from It\^o's stochastic contraction; therefore in Stratanovich form 
\begin{equation}\label{eq4.11}
\circ d_tf_t = -\Bigl(\int_\T f_t a_t\, \rho dx\Bigr)\hat\rho_t(X_t)\circ dB_t.
\end{equation}
Now by It\^o\textcolor{black}{'s } formula, 
\begin{equation}\label{eq4.12}
\circ d_t(f_t\tilde K_t)=- \Bigl(\int_\T f_t a_t\, \rho dx\Bigr)\hat\rho_t(X_t) \tilde K_t\circ dB_t
+f_t\, \partial_x^2\phi_t(X_t)\, \tilde K_t\circ dB_t.
\end{equation}
Note that $\dis \int_\T \hat\rho_t(X_t)\tilde K_t\, dx=\int_\T \hat\rho_t(x)\, dx=1$. On the other hand, 
\begin{equation*}
\int_\T f_t\, \partial_x^2\phi_t(X_t)\, \tilde K_t\, dx
=\int_\T f_t\times \frac{\partial_x^2\phi_t}{\rho_t}(X_t)\, \rho_t(X_t)\,\tilde K_t\, dx
=\int_\T f_t a_t\, \rho\, dx,
\end{equation*}
the last equality being due to  $\dis \rho_t(X_t)\,\tilde K_t=\rho$ by Lemma \ref{LEMMA}.  Now 
we get $\circ d_t\int_\T f_t\tilde K_t dx=0$ using  \eqref{eq4.12}.
Therefore $\dis \int_\T g_t dx=\int_\T g_0 dx=0$.
\end{proof}

\begin{theorem}\label{th4.5} We have, 
for any $t\in [0,1]$, 
\begin{equation}\label{eq4.13}
\int_\T |g_t(x)|^2\rho_t(x)\, dx=\int_\T |g_0(x)|^2\, \rho dx.
\end{equation}
\end{theorem}

\vskip 2mm
\begin{proof} By \eqref{eq4.11}, 
\begin{equation*}
\begin{split}
\circ d_t \int_\T f_t^2\, \rho dx=&- 2\Bigl[ \Bigl( \int_\T f_ta_t\rho dx\Bigr) \hat\rho_t(X_t)\, f_t\, \rho dx\Bigr]\circ dB_t\\
&=- 2\Bigl[ \Bigl( \int_\T f_ta_t\rho dx\Bigr) \Bigl( \int_\T \hat\rho_t(X_t) f_t(x)\, dx\Bigr)\Bigr]\, \circ dB_t,
\end{split}
\end{equation*}
But we have seen that 
\begin{equation*}
 \int_\T \hat\rho_t(X_t) f_t(x)\, dx=\Bigl(\int_\T g_t(x)dx\Bigr)\Bigl(\int_\T \frac{dx}{\rho_t}\Bigr)^{-1}
\end{equation*}
which is equal to $0$ by Proposition \ref{prop4.4}. 
\end{proof}

Combining all above results, finally we get

\begin{theorem}\label{th4.6}
  Let $\partial_x\Psi_t=g_t$. Then  for $\mu=\rho\, dx$ and $\mu_t=(X_t)_\#(\rho dx)$, 
$\{\partial_x\Psi_t;\ t\in [0,1]\}$ is the parallel translation along the stochastic regular curve $\{\mu_t;\ t\in [0,1]\}$, that is,
$\dis \partial_x\Psi_t\in \TT_{\mu_t}$ and 
\begin{equation*}
 \int_{\T} |\partial_x\Psi_t|^2\, \mu_t(dx)=\int_{\T}|\partial_x\Psi_0|^2\, \rho dx,\quad t\in [0,1].
 \end{equation*}
\end{theorem}

\vskip 3mm
Actually, it is well-known that some quasi-invariant non-degenerated diffusion processes have been constructed
in $\P_2(\T)$, see for example \cite{RS1, Wang}. It seems that these diffusion processes do not charge the subspace 
$\P_{2,\infty}(\T)$. In what follows, we construct a non-degenerated diffusion process $\{\mu_t; t\in [0,1]\}$ on 
$\P_{2,\infty}(\T)$ and parallel translations along it. 

\vskip 2mm

For $k\in \N^*$, set 
\begin{equation*}
\phi_{2k-1}(x)=\frac{\sin(kx)}{k},\quad \phi_{2k}(x)=-\frac{\cos(kx)}{k},\quad\hbox{\rm and }\quad a_k=k^q.
\end{equation*}

For an integer $N\geq 2$, we consider the SDE on $\T$, 
\begin{equation*}
dX_t^N=\sum_{k=1}^N \frac{1}{\alpha_k} \Bigl( \partial_x\phi_{2k-1}(X_t^N)\circ dB_{2k-1}(t)
+   \partial_x\phi_{2k}(X_t^N)\circ dB_{2k}(t)\Bigr).
\end{equation*}

Note that $\dis d_t \partial_x\phi_{2k-1}(X_t^N)\cdot dB_{2k-1}+  d_t \partial_x\phi_{2k}(X_t^N)\cdot dB_{2k}=0$; therefore 
the above Stratanovich SDE becomes the below It\^o SDE:

\begin{equation}\label{eq4.14}
dX_t^N=\sum_{k=1}^N \frac{1}{\alpha_k} \Bigl( \partial_x\phi_{2k-1}(X_t^N)\, dB_{2k-1}(t)
+   \partial_x\phi_{2k}(X_t^N)\, dB_{2k}(t)\Bigr).
\end{equation}

It is well-known (see \cite{Elworthy, IW, Kunita, Malliavin}, especially in \cite{AR, Fang}) that for $q>2$, almost surely,  as $N\ra +\infty$, 
$\dis X_t^N$ converges in $\dis C\bigl([0,1], \hbox{\rm Diff}(\T)\bigr)$ to $X_t$, which solves the following SDE: 

\begin{equation}\label{eq4.15}
dX_t=\sum_{k=1}^\infty \frac{1}{\alpha_k} \Bigl( \partial_x\phi_{2k-1}(X_t)\, dB_{2k-1}(t)
+   \partial_x\phi_{2k}(X_t)\, dB_{2k}(t)\Bigr).
\end{equation}

\vskip 2mm
Let $d\mu=\rho\, dx\in \P_{2,\infty}(\T)$ be given; for any $N>2$, we denote by $\rho_t^N$ 
the density of the measure $\dis (X_t^N)_\#(\mu)$ with respect to $dx$. It is obvious that almost surely, $\rho_t^N$ 
converges to the density $\rho_t$ of  $\dis (X_t)_\#(\mu)$ uniformly in $(t,x)\in [0,1]\times\T$. It is quite automatic that 
results in Theorem \ref{th4.6} remain valid for SDE \eqref{eq4.14}. More precisely, for any $\partial_x\Psi_0$ given in 
$\dis L^2(\rho\,dx)$, there exists the parallel translation $\{\partial_x\Psi_t^N;\ t\in [0,1]\}$ along $\{\mu_t^N;\ t\in [0,1]\}$, 
that is, $\partial_x\Psi_t\in \TT_{\mu_t^N}$ and 
\begin{equation*}
\int_\T |\partial_x\Psi_t^N|^2\, \rho_t^N\,dx=\int_\T |\partial_x\Psi_0|^2\, \rho\,dx.
\end{equation*}

\vskip 2mm
For simplicity, we again use the notation $\dis g_t^N=\partial_x\Psi_t^N$ and $\dis f_t^N=g_t^N(X_t^N)$. 
We introduce 
\begin{equation*}
a_k^N(t)=\Bigl( \frac{\partial_x^2\phi_k}{\rho_t^N}\Bigr)(X_t^N), \quad
b_k^N(t)=\Bigl( \frac{\partial_x\bigl(\partial_x^2\phi_k\,\partial_x\phi_k\bigr)}{\rho_t^N}\Bigr)(X_t^N),
\end{equation*}
and 
\begin{equation*}
\Lambda_k^N(t,f)=-\Bigl(\int_\T f\, a_k^N(t)\, \rho\,dx\Bigr)\, \hat\rho_t^N(X_t^N), 
\end{equation*}
\begin{equation*}
\begin{split}
\Theta_k^N(t,f)=-&\Bigl(\int_\T f\, a_k^N(t)\, \rho\,dx\Bigr)\, \bigl(\hat\rho_t^N\,\partial_x^2\phi_k\bigr)(X_t^N)
-\Bigl(\int_\T f\, b_k^N(t)\, \rho\,dx\Bigr)\, \hat\rho_t^N(X_t^N)\\
+& 3\ \Bigl(\int_\T f\, a_k^N(t)\, \rho\,dx\Bigr)\, \Bigl(\int_\T \partial_x^2\phi_k\, \hat\rho_t^N\, dx\Bigr)\hat\rho_t^N(X_t^N).
\end{split}
\end{equation*}

Then by \eqref{eq4.10}, $f_t^N$ satisfies the following SDE 
\begin{equation}\label{eq4.16}
d_tf_t^N=\sum_{k=1}^{2N} \frac{1}{\alpha_k}\, \Lambda_k^N(t, f_t^N)\, dB_t^k
+\sum_{k=1}^{2N} \frac{1}{2\alpha_k^2}\,\Theta_k^N(t,f_t^N)\, dt,
\end{equation}
with $\dis\alpha_k = [\frac{k+1}{2}]^q$. In the sequel, we will use the notation: 
$\dis \xi(s)=\sum_{k=1}^{+\infty} \frac{1}{k^s}$ for $s>1$.

\begin{theorem}\label{th4.8} Let $\dis q>\frac{5}{2}$; then
almost surely, as $N\ra +\infty$, $f_\cdot^N$ converges in $\dis C\bigl([0,1], L^2(\T, \rho\,dx)\bigr)$. More precisely,
there exists $f\in C\bigl([0,1], L^2(\T, \rho\,dx)\bigr)$ such that 
\begin{equation}\label{eq4.17}
\lim_{N\ra +\infty} \sup_{t\in [0,1]}\int_\T |f_t^N-f_t|^2\rho\, dx =0.
\end{equation}
\end{theorem}

\begin{proof} We first remark that 
\begin{equation*}
||\partial_x^2\phi_k||_\infty \leq k,\quad ||\partial_x(\partial_x^2\phi_k\, \partial_x\phi_k)||_\infty\leq k^2.
\end{equation*}
Then, by Lemma \ref{lemma4.4}, for any $N\geq 2$, 
\begin{equation}\label{eq4.18}
||\Lambda_k^N(t,f)||_{L^2(\rho dx)}\leq k\, ||f||_{L^2(\rho dx)},\quad
||\Theta_k^N(t,f)||_{L^2(\rho dx)}\leq 3k^2\, ||f||_{L^2(\rho dx)}.
\end{equation}

Using \eqref{eq4.18}, we have, 
\begin{equation*}
\sum_{k=1}^{+\infty} \Bigl\| \frac{1}{\alpha_k}\, \Lambda_k^N(t,f)\Bigr\|^2_{L^2(\rho dx)}
\leq \xi(2q-2)\, ||f||_{L^2(\rho dx)},
\end{equation*}
as well as
\begin{equation*}
\sum_{k=1}^{+\infty} \Bigl\| \frac{1}{\alpha_k^2}\Theta_k^N(t,f)\Bigr\|_{L^2(\rho dx)}
\leq 3\xi(2q-2)\, ||f||_{L^2(\rho dx)}.
\end{equation*}

Combining these estimates, together with SDE \eqref{eq4.16}, we get that, for any $p\geq 1$, 
\begin{equation}\label{eq4.21}
D_p=\sup_N\sup_{t\in [0,1]} \E\Bigl(||f_t^N||_{L^2(\rho dx)}^p\Bigr) <+\infty.
\end{equation}

Now let $N'>N\geq 2$, we have
\begin{equation*}
\begin{split}
d_t\bigl(f_t^{N'}-f_t^N\bigr)=& \sum_{k=1}^{2N} \frac{1}{\alpha_k}\Bigl( \Lambda_k^{N'}(t,f_t^{N'})- \Lambda_k^{N}(t,f_t^{N})\Bigr)\, dB_t^k
\\&+ \sum_{k=1}^{2N} \frac{1}{2\alpha_k^2} \Bigl( \Theta_k^{N'}(t,f_t^{N'})- \Theta_k^{N}(t,f_t^{N})\Bigr)\ dt\\
&\hskip -10mm + \sum_{k=2N+1}^{2N'} \frac{1}{\alpha_k} \Lambda_k^{N'}(t,f_t^{N'})\, dB_t^k
 +  \sum_{k=2N+1}^{2N'} \frac{1}{2\alpha_k^2} \Theta_k^{N'}(t,f_t^{N'})\, dt\\
 &=dI_1(t)+dI_2(t)+dI_3(t)+dI_4(t)\quad\hbox{\rm respectively}.
\end{split}
\end{equation*}
By \textcolor{black}{the} B\"urkh\"older inequality, 
\begin{equation*}
\E\Bigl(\sup_{s\in [0,t]}||I_3(s)||_{L^2(\rho dx)}^2\Bigr)
\leq 4\ \sum_{k>2N} \E\Bigl( \int_0^t \frac{1}{\alpha_k^2}||\Lambda_k^{N'}(s, f_s^{N'})||_{L^2(\rho dx)}^2\, ds\Bigr),
\end{equation*}
which is dominated, using \eqref{eq4.18}, 
\begin{equation*}
4 \sum_{k>2N} \frac{k^2}{k^{2q}}\, \E\Bigl(\int_0^t ||f_s^{N'}||_{L^2(\rho dx)}^2\, ds\Bigr)
\leq \frac{4}{(2N)^{2q-3}}\, \E\Bigl(\int_0^1 ||f_s^{N'}||_{L^2(\rho dx)}^2\Bigr).
\end{equation*}

According to \eqref{eq4.21}, we get
\begin{equation}\label{eq4.22}
\E\Bigl(\sup_{s\in [0,t]}||I_3(s)||_{L^2(\rho dx)}^2\Bigr)
\leq \frac{4D_1}{(2N)^{2q-3}}.
\end{equation}
Secondly,  according to second estimate in \eqref{eq4.18},  we have
\begin{equation*}
\begin{split}
&||I_4(t)||_{L^2(\rho dx)}\leq \sum_{k=2N+1}^{2N'} \frac{1}{2\alpha_k^2}\, \int_0^t ||\Theta_k^{N'}(s, f_s^{N'})||_{L^2(\rho dx)}\,ds\\
& \leq  \sum_{k=2N+1}^{2N'} \frac{3k^2}{2 k^{2q}}\,\int_0^t ||f_s^{N'}||_{L^2(\rho dx)}\, ds
\leq \frac{6}{(2N)^{2q-3}}\,\int_0^t ||f_s^{N'}||_{L^2(\rho dx)}\, ds.
\end{split}
\end{equation*}

It follows that, for $N$ big enough, 
\begin{equation}\label{eq4.23}
\E\Bigl( \sup_{s\in [0,t]} ||I_4(s)||_{L^2(\rho dx)}^2\Bigr) 
\leq \Big(  \frac{6}{(2N)^{2q-3}}\Bigr)^2\, D_1\leq  \frac{4D_1}{(2N)^{2q-3}}. 
\end{equation}

Estimating  $I_1(t)$ and $I_2(t)$ will make appear Gronwall type inequality for 
\begin{equation*}
 \E\Bigl( \sup_{s\in [0,t]}||f_s^{N'}-f_s^N||_{L^2(\rho dx)}^2\Bigr),
 \end{equation*} 
which will yield desired result. 

\vskip 2mm
Again by B\"urkh\"older inequality, we have 
\begin{equation}\label{eq4.24}
\E\Bigl( \sup_{s\in [0,t]}||I_1(s)||_{L^2(\rho dx)}^2\Bigr)\leq
\sum_{k=1}^{2N} \frac{4}{\alpha_k^2} \E\Bigl(\int_0^t ||\Lambda_k^{N'}(s, f_s^{N'})-\Lambda_k^{N}(s, f_s^{N})||_{L^2(\rho dx)}^2\, ds\Bigr).
\end{equation}

By expression of $\Lambda_k$, we write down
\begin{equation}\label{eq4.24.1}
\begin{split}
\Lambda_k^{N'}(s, f_s^{N'})-\Lambda_k^{N}(s, f_s^{N})
=& \Bigl[ -\int_\T f_s^{N'}a_k^{N'}(s)\, \rho dx+ \int_\T f_s^{N}a_k^{N}(s)\, \rho dx\Bigr]\, \hat\rho_s^{N'}(X_s^{N'})\\
&+\Bigl( \int_\T f_s^{N}a_k^{N}(s)\, \rho dx\Bigr)\, \Bigl( \hat\rho_s^{N}(X_s^{N}) - \hat\rho_s^{N'}(X_s^{N'})\Bigr)\\
&=J_1(s)+J_2(s)\quad\hbox{respectively.}
\end{split}
\end{equation}

As for getting estimates in \eqref{eq4.18}, we have 
\begin{equation}\label{eq4.25}
\Bigl\|\Bigl( \int_\T (f_s^N-f_s^{N'})a_k^{N'}(s)\, \rho dx\Bigr)\, \hat\rho_s^{N'}(X_s^{N'})\Bigr\|_{L^2(\rho dx)}
\leq k\, ||f_s^N-f_s^{N'}||_{L^2(\rho dx)}.
\end{equation}
Hence $\dis ||\Lambda_k^{N'}(s, f_s^{N'})-\Lambda_k^{N}(s, f_s^{N})||_{L^2(\rho dx)}^2$ is dominated by 
the sum of two terms in the following way

\begin{equation*}
2k^2\, ||f_s^N-f_s^{N'}||_{L^2(\rho dx)}^2 + o_k(s, N, N'),
\end{equation*}

so that the right hand side of \eqref{eq4.24} has the following upper bound
\begin{equation*}
\sum_{k=1}^{2N} \frac{8k^2}{k^{2q}}\,  \int_0^t \E\Bigl( ||f_s^N-f_s^{N'}||^2_{L^2(\rho dx)}\Bigr)\, ds
+\sum_{k=1}^{2N} \frac{4}{k^{2q}}\, \int_0^t \E(o_k(s,N,N'))\, ds.
\end{equation*} 
Therefore the right hand side of \eqref{eq4.24} is dominated by 
\begin{equation*}
8\xi(2q-2)\,  \int_0^t \E\Bigl( ||f_s^N-f_s^{N'}||^2_{L^2(\rho dx)}\Bigr)\, ds
+o(N).
\end{equation*} 
We have in fact the following inequality
\begin{equation}\label{eq4.26}
\E\Bigl( \sup_{s\in [0,t]}||I_1(s)||_{L^2(\rho dx)}^2\Bigr)\leq
8\xi(2q-2)\,  \int_0^t \E\Bigl( ||f_s^N-f_s^{N'}||^2_{L^2(\rho dx)}\Bigr)\, ds
+\frac{C}{N^{q-\frac{1}{2}}}.
\end{equation} 

Since the detail of the proof of \eqref{eq4.26} is lengthy, we will do it in the following proposition. In the same way, there is a constant \textcolor{black}{$C_2>0$} such that 

\begin{equation}\label{eq4.27}
\E\Bigl( \sup_{s\in [0,t]}||I_2(s)||_{L^2(\rho dx)}^2\Bigr) 
\leq \textcolor{black}{C_2\xi(2q-2)}\,  \int_0^t \E\Bigl( ||f_s^N-f_s^{N'}||^2_{L^2(\rho dx)}\Bigr)\, ds
+ \frac{C}{N^{q-\frac{1}{2}}}.
\end{equation}
Also the proof of \eqref{eq4.27} is postponed in the following proposition.  
For $\dis q>\frac{5}{2}$, $2q-3>q-\frac{1}{2}$, \textcolor{black}{combining}
\eqref{eq4.22}, \eqref{eq4.23}, \eqref{eq4.26} and \eqref{eq4.27}, we finally get 
\begin{equation*}
\E\Bigl( \sup_{s\in [0,t]}||f_s^N-f_s^{N'}||_{L^2(\rho dx)}^2\Bigr)
\leq C_1\, \int_0^t \E\Bigl( ||f_s^N-f_s^{N'}||_{L^2(\rho dx)}^2\Bigr)\, ds + C_2 N^{-(q-\frac{1}{2})}.
\end{equation*}
By Gronwall\textcolor{black}{'s} lemma, there is a constant $C>0$ such that, for any $N'\geq N$, 

\begin{equation*}\label{eq4.28}
\E\Bigl( \sup_{s\in [0,t]}||f_s^N-f_s^{N'}||_{L^2(\rho dx)}^2\Bigr)
\leq  C\, N^{-(q-\frac{1}{2})} \textcolor{black}{e^{C_1t}}.
\end{equation*}

It follows that $\{f^N; N\geq 2\}$ is a Cauchy sequence in $\dis L^2\Bigl(\Omega, C\bigl([0,1], L^2(\rho dx)\bigr)\Bigr)$;  then 
there exists $f\in  L^2\Bigl(\Omega, C\bigl([0,1], L^2(\rho dx)\bigr)\Bigr)$ such that 
\begin{equation}\label{eq4.28}
\E\Bigl( \sup_{s\in [0,1]}||f_s^N-f_s||_{L^2(\rho dx)}^2\Bigr)
\leq  C\, N^{-(q-\frac{1}{2})}.
\end{equation}
 This estimate is useful in order to obtain the almost surely convergence. Let $\beta \in (0, 1/2)$, we put 
 \begin{equation*}
 \Omega_N=\Bigl\{\omega;\ \sup_{t\in [0,1]}||f_t^N-f_t||_{L^2(\rho dx)}\geq N^{-\beta}\Bigr\}.
 \end{equation*}
By \eqref{eq4.28}, $\dis {\mathbf P}(\Omega_N)\leq C\, N^{-(q-(1/2)-2\beta)}$. For $\dis q>\frac{5}{2}$ 
and $\beta\in (0,1/2)$, the series $\dis \sum_{N>1} N^{-(q-(1/2)-2\beta)}$ converges;
 therefore the Borel-Cantelli lemma yields the almost sure convergence 
 with a convergence rate $N^{-\beta}$. 
\end{proof}

\begin{proposition}\label{prop4.9} Above inequality \eqref{eq4.26} as well as inequality \eqref{eq4.27} hold true.
\end{proposition}

\begin{proof}
Recall the expression of $J_1(s)$ in \eqref{eq4.24.1}:
\begin{equation*}
J_1(s)= \Bigl[ -\int_\T f_s^{N'}a_k^{N'}(s)\, \rho dx+ \int_\T f_s^{N}a_k^{N}(s)\, \rho dx\Bigr]\, \hat\rho_s^{N'}(X_s^{N'}).
\end{equation*}
Put
\begin{equation*}
J_{11}(s)= \Bigl[ \int_\T \bigl(f_s^N-f_s^{N'}\bigr)a_k^{N'}(s)\, \rho dx\Bigr]\,\hat\rho_s^{N'}(X_s^{N'}),
\end{equation*}
and 
\begin{equation*}
J_{12}(s)= \Bigl[ \int_\T  f_s^N\bigl( a_k^N(s)-a_k^{N'}(s)\bigr)\, \rho dx\Bigr]\,\hat\rho_s^{N'}(X_s^{N'}).
\end{equation*}
Then $\dis J_1(s)=J_{11}(s)+J_{12}(s)$. Obviously 
\begin{equation}\label{eq4.29.1}
||J_{11}(s)||_{L^2(\rho dx)}\leq k\, ||f_s^N-f_s^{N'}||_{L^2(\rho dx)}.
\end{equation}
Remark that for any \textcolor{black}{strictly} positive probability density $\sigma$, $\dis \int_\T \frac{dx}{\sigma}\geq 1$, since 
\begin{equation*}
1=\int_\T \sqrt{\sigma}\cdot \frac{1}{\sqrt{\sigma}}\,dx
\leq \Bigl( \int_\T \frac{dx}{\sigma}\Bigr)^{1/2}.
\end{equation*}
Therefore 
\begin{equation*}
\int_\T \bigl(\hat\rho_s^N\bigr)^2(X_s^{N'})\, \rho dx=1/\int_\T\frac{dx}{\rho_s^{N'}}\leq 1.
\end{equation*}
So by expression of $J_{12}(s)$ and Cauchy-Schwarz inequality, we get 
\begin{equation}\label{eq4.29.2}
\int_\T |J_{12}(s)|^2\, \rho dx \leq ||f_s^N||_{L^2(\rho dx)}^2\, \int_\T \Bigl( a_k^N(s)-a_k^{N'}(s)\Bigr)^2\, \rho dx.
\end{equation}

Now 
\begin{equation*}
a_k^N(s)-a_k^{N'}(s)
=\frac{\partial_x^2\phi_k(X_s^N)-\partial_x^2\phi_k(X_s^{N'})}{\rho_s^{N'}(X_s^{N'})}
+ \partial_x^2\phi_k(X_s^N)\Bigl(\frac{1}{\rho_s^{N}(X_s^{N})} -\frac{1}{\rho_s^{N'}(X_s^{N'})}\Bigr).
\end{equation*}
Remark that
\begin{equation*}
\frac{|\partial_x^2\phi_k(X_s^N)-\partial_x^2\phi_k(X_s^{N'})|}{\rho_s^{N'}(X_s^{N'})}
\leq k^2\, \frac{ |X_s^N-X_s^{N'}|}{\rho_s^{N'}(X_s^{N'})}.
\end{equation*}
Put $\dis V_k(s)= \int_\T \Bigl( a_k^N(s)-a_k^{N'}(s)\Bigr)^2\, \rho dx$, then, according to above two relations, 
\begin{equation*}
V_k(s) \leq 2\,\Bigl[ k^4 \int_\T\frac{ |X_s^N-X_s^{N'}|^2}{(\rho_s^{N'}(X_s^{N'}))^2}\, \rho dx
+k^2 \int_{\T} \Bigl(\frac{1}{\rho_s^{N}(X_s^{N})} -\frac{1}{\rho_s^{N'}(X_s^{N'})}\Bigr)^2\, \rho dx\Bigr].
\end{equation*}
Putting
\begin{equation*}
V_{k1}(s)=\int_\T\frac{ |X_s^N-X_s^{N'}|^2}{(\rho_s^{N'}(X_s^{N'}))^2}\, \rho dx,\quad
V_{k2}(s)=\int_{\T} \Bigl(\frac{1}{\rho_s^{N}(X_s^{N})} -\frac{1}{\rho_s^{N'}(X_s^{N'})}\Bigr)^2\, \rho dx,
\end{equation*}
then $\dis V_k(s) \leq  2 k^4 V_{k1}(s)+ 2 k^2 V_{k2}(s)$. Using \eqref{eq4.29.2}, we write down 

\begin{equation}\label{eq4.29.3}
\int_\T |J_{12}(s)|^2\, \rho dx\leq 2k^4 ||f_s^N||_{L^2(\rho dx)}^2\, V_{k1}(s)+ 2 k^2 ||f_s^N||_{L^2(\rho dx)}^2 V_{k2}(s).
\end{equation}

Let $r>1$ and $\tilde r>1$ such that $\dis \frac{1}{r}+ \frac{1}{\tilde r}=1$. By H\"older inequality, we have
\begin{equation*}
\int_0^t \E\Bigl( ||f_s^N||_{L^2}^2 V_{k1}(s)\Bigr)\,ds
\leq \Bigl[ \int_0^t \E\Bigl( ||f_s^N||_{L^2}^{2\tilde r}\Bigr)\,ds\Bigr]^{1/\tilde r}
\Bigl[\int_0^t \E\Bigl(\int_\T\frac{ |X_s^N-X_s^{N'}|^{2r}}{(\rho_s^{N'}(X_s^{N'}))^{2r}}\, \rho dx\Bigr)\,ds\Bigr]^{1/r}.
\end{equation*}
Denote by $\dis o_N(t)^{1/r}$ the last term in the product in the right hand, again by  H\"older inequality, 
we have
\begin{equation}\label{eq4.29.4}
o_N(t)\leq \Bigl[\int_0^t\int_\T \E\bigl( |X_s^N-X_s^{N'}|^{2r^2}\bigr)\, \rho dx ds\Bigr]^{1/r}
\Bigl[\int_0^t\int_\T \E\Bigl( \frac{1}{\bigl(\rho_s^{N'}(X_s^{N'})\bigr)^{2r\tilde r}}\Bigr)\, \rho dx ds\Bigr]^{1/\tilde r}.
\end{equation}
Now using \eqref{eq5.3} in the next section and for $N'\geq N$, we have 
\begin{equation*}
\Bigl( \E\Bigl[ |X_s^N-X_s^{N'}|^{2r^2}\Bigr]\Bigr)^{1/2r^2}
\leq 2\,  \Bigl( \frac{C_{r^2}}{N^{2q-1}} \Bigr)^{1/r^2};
\end{equation*}
therefore for some constant $C>0$, independent of $\omega, x, t$ such that 
\begin{equation}\label{eq4.29.5}
\Bigl[\int_0^t\int_\T \E\bigl( |X_s^N-X_s^{N'}|^{2r^2}\bigr)\, \rho dx ds\Bigr]^{1/r^2}
\leq \frac{C}{N^{(2q-1)/r^2}}=\frac{C}{N^{q-\frac{1}{2}}},
\end{equation}
if we take $r^2=2$. For estimating the second factor on the right hand of \eqref{eq4.29.4}, we introduce 
$\tilde K_s^N$ the density of $(X_s^N)^{-1}_\#(dx)$ relative to $dx$. By Lemma \ref{LEMMA}, 
$\dis \rho_s^{N'}(X_s^{N'})\, \tilde K_s^N=\rho$, which implies that 
\begin{equation*}
\frac{1}{(\rho_s^{N'}(X_s^{N'}))^{2r\tilde r}}=\frac{(\tilde K_s^{N'})^{2r\tilde r}}{\rho^{2r\tilde r}}.
\end{equation*}
Since $0<\delta_1\leq \rho\leq \delta_2$ and  using \eqref{eq5.6} in Appendice, we get 
\begin{equation*}
\sup_N \sup_{s\in [0,1]} \E\Bigl(\int_\T \frac{(\tilde K_s^{N})^{2r\tilde r}}{\rho^{2r\tilde r}}\rho dx\Bigr) <+\infty.
\end{equation*}
Now combining this result with \eqref{eq4.29.4}, \eqref{eq4.29.5}, we get 

\begin{equation}\label{eq4.29.6}
\int_0^t \E\Bigl( ||f_s^N||_{L^2}^2 V_{k1}(s)\Bigr)\,ds\leq \frac{C k^4}{N^{q-\frac{1}{2}}}.
\end{equation}
Again by Lemma \ref{LEMMA}, we have 
$\dis V_{k2}(s)=\int_\T (\tilde K_s^N-\tilde K_s^{N'})^2\, \rho^{-1}\, dx$, so that  

\begin{equation*}
\int_0^t\E\Bigl(||f_s^N||_{L2}^2 V_{k2}(s)\Bigr)\, ds
\leq \Bigl(\int_0^t \E\Big[||f_s^N||_{L^2}^{2\tilde r}\Bigr] ds\Bigr)^{1/\tilde r}
\Bigl(\int_0^t\int_\T \E\Bigl[|\tilde K_s^N-\tilde K_s^{N'}|^{2r}\Bigr]\rho^{-r+1} dxds\Bigr)^{1/r}.
\end{equation*}

By Proposition \ref{prop5.2} in Appendice, for some $\delta>0$ such that $\dis q-\frac{1}{2}<q-\delta<2q-3$, 
\begin{equation*}
\E\Bigl[|\tilde K_s^N-\tilde K_s^{N'}|^{2r}\Bigr]
\leq \frac{C}{N^{q-\delta}}.
\end{equation*}
For some $r>1$ close to $1$, we have $\dis \frac{1}{N^{(q-\delta)/r}}=\frac{1}{N^{q-\frac{1}{2}}}$. 
Finally we get
\begin{equation}\label{eq4.29.7}
\int_0^t\E\Bigl(||f_s^N||_{L2}^2 V_{k2}(s)\Bigr)\, ds\leq \frac{C}{N^{q-\frac{1}{2}}}.
\end{equation}
Finally combining \eqref{eq4.29.3}, \eqref{eq4.29.6} and \eqref{eq4.29.7}, we obtain 
\begin{equation}\label{eq4.29.8}
\int_0^t\E\Bigl[\int_{\T} J_{12}(s)^2\, \rho dx\Bigr]\,ds \leq \frac{Ck^4}{N^{q-\frac{1}{2}}}.
\end{equation}

For  estimating $J_2(s)$, we remark that 
\begin{equation*}
\begin{split}
\hat\rho_s^N(X_s^N)-\hat\rho_s^{N'}(X_s^{N'})
&= \Bigl( \frac{1}{\rho_s^N(X_s^N)}-\frac{1}{\rho_s^{N'}(X_s^{N'})}\Bigr) \frac{1}{\int_\T \frac{dx}{\rho_s^{N}}}
+\frac{1}{\rho_s^{N'}(X_s^{N'})}\Bigl(\frac{1}{\int_\T \frac{dx}{\rho_s^{N}}}-\frac{1}{\int_\T \frac{dx}{\rho_s^{N'}}}\Bigr)
\\
&= \frac{\tilde K_s^N-\tilde K_s^{N'}}{\rho \int_\T \frac{dx}{\rho_s^{N}}}
+\frac{1}{\rho_s^{N'}(X_s^{N'})}\, \frac{\int_\T \frac{dx}{\rho_s^{N'}}
-\int_\T \frac{dx}{\rho_s^{N}}}{\int_\T \frac{dx}{\rho_s^{N'}}\int_\T \frac{dx}{\rho_s^{N'}}}.
\end{split}
\end{equation*}
Note that 
\begin{equation*}
\int_\T \frac{dx}{\rho_s^{N}}=\int_\T \frac{\rho dx}{(\rho_s^N(X_s^N))^2}
=\int_\T \frac{(\tilde K_s^N)^2}{\rho}\, dx;
\end{equation*}
then
\begin{equation*}
\int_\T \frac{dx}{\rho_s^{N'}}-\int_\T \frac{dx}{\rho_s^{N}}
=\int_\T \frac{(\tilde K_s^{N'})^2- (\tilde K_s^N)^2}{\rho}\, dx.
\end{equation*}
On the other hand,
\begin{equation*}
\int_\T (a_k^N(s))^2\, \rho dx \leq k^2 \int_\T \frac{\rho dx}{(\rho_s^N(X_s^N))^2}
=k^2 \int_\T \frac{dx}{\rho_s^{N}},
\end{equation*}
so that 
\begin{equation*}
\Bigl|\int_\T f_s^N a_k^N(s)\, \rho dx\Bigr|\leq k ||f_s^N||_{L^2} \Bigl(\int_\T \frac{dx}{\rho_s^{N}} \Bigr)^{1/2}.
\end{equation*}
Recall that 
\begin{equation*}
J_2(s)=\Bigl(\int_\T f_s^N a_k^N(s)\, \rho dx \Bigr) \Bigl(\hat\rho_s^N(X_s^N)-\hat\rho_s^{N'}(X_s^{N'}) \Bigr).
\end{equation*}
According to above calculation, we get
\begin{equation*}
|J_2(s)|\leq k ||f_s^N||_{L^2}\, \frac{|\tilde K_s^N-\tilde K_s^{N'}|}{\rho}
+\frac{ k ||f_s^N||_{L^2}}{\rho_s^{N'}(X_s^{N'})}\,
\Bigl|\int_\T \frac{(\tilde K_s^{N'})^2- (\tilde K_s^N)^2}{\rho}\, dx\Bigr|.
\end{equation*}
Proceeding as for estimating $J_1(s)$, we get finally
\begin{equation}\label{eq4.29.9}
\int_0^t \E\Bigl[ |J_2(s)|^2\rho dx\Bigr]\, ds \leq \frac{C k^2}{N^{q-\frac{1}{2}}}.
\end{equation}
Now using \eqref{eq4.29.1}, \eqref{eq4.29.8} and \eqref{eq4.29.9} and 
the fact that the series $\dis \sum_{k\geq 1}\frac{k^4}{k^{2q}}$ converges for $\dis q>\frac{5}{2}$,
 we get \eqref{eq4.26}. 
Inequality \eqref{eq4.27} can be proved in a similar way. 

\end{proof}

\begin{definition}\label{DEF}
We define $\dis g_t=f_t(X_t^{-1})$.
\end{definition}

 It is obvious that  $\dis \int_\T g_t^2\,\rho_t\, dx<+\infty$. In waht follows, we justify that $\{g_t; t\in [0,1]\}$ is the 
 stochastic parallel translation along the curve $\{(X_t)_\#(\rho dx); \ t\in [0,1]\}$. We have no explicit SDE for $g_t$, 
 but $f_t=g_t(X_t)$ satisfies a SDE. More precisely, $\{f_t; t\in [0,1]\}$ is a solution to the following SDE, with $\dis q>\frac{5}{2}$,
 
 \begin{equation}\label{eq4.29.10}
d_tf_t=\sum_{k=1}^{+\infty} \frac{1}{k^q}\, \Lambda_k(t, f_t)\, dB_t^k
+\sum_{k=1}^{+\infty} \frac{1}{2 k^{2q}}\,\Theta_k(t,f_t)\, dt,
\end{equation}
where
\begin{equation*}
\Lambda_k(t,f)=-\Bigl(\int_\T f\, a_k(t)\, \rho\,dx\Bigr)\, \hat\rho_t(X_t), 
\end{equation*}
\begin{equation*}
\begin{split}
\Theta_k(t,f)=-&\Bigl(\int_\T f\, a_k(t)\, \rho\,dx\Bigr)\, \bigl(\hat\rho_t\,\partial_x^2\phi_k\bigr)(X_t)
-\Bigl(\int_\T f\, b_k(t)\, \rho\,dx\Bigr)\, \hat\rho_t(X_t)\\
+& 3\ \Bigl(\int_\T f\, a_k(t)\, \rho\,dx\Bigr)\, \Bigl(\int_\T \partial_x^2\phi_k\, \hat\rho_t\, dx\Bigr)\hat\rho_t(X_t), 
\end{split}
\end{equation*}
with 
\begin{equation*}
a_k(t)=\Bigl( \frac{\partial_x^2\phi_k}{\rho_t}\Bigr)(X_t), \quad
b_k(t)=\Bigl( \frac{\partial_x\bigl(\partial_x^2\phi_k\,\partial_x\phi_k\bigr)}{\rho_t}\Bigr)(X_t).
\end{equation*}

\begin{corollary}\label{corollary4.9}  Under hypothesis of theorem \ref{th4.8}, almost surely, 
\begin{equation}\label{eq4.29}
\lim_{N\ra +\infty} \int_\T (g_t^N)^2\, \rho_t^N\, dx=\int_\T g_t^2\, \rho_t\, dx.
\end{equation}
\end{corollary}
\begin{proof}
We have
\begin{equation*}
\int_\T (g_t^N)\rho_t^n\, dx=\int_\T (f_t^N)^2\, \rho dx
\end{equation*}
which converges, by \eqref{eq4.17}, to 
\begin{equation*}
\int_\T f_t^2\, \rho dx=\int_\T g_t^2\, \rho_t\, dx.
\end{equation*}
\end{proof}

An immediate consequence of above result is 
\begin{corollary}\label{corollary4.10}
For any $t\in [0,1]$, 
\begin{equation}\label{eq4.30}
\int_\T g_t^2\, \rho_t\, dx=\int_\T g_0^2\, \rho dx.
\end{equation}
\end{corollary}

\begin{theorem}\label{th4.11} Almost surely, for $t\in [0,1]$, 
\begin{equation}\label{eq4.31}
\int_\T g_t(x)\, dx=0.
\end{equation}
\end{theorem}

\begin{proof}  Let $\xi$ be a bounded Random variable, using the uniform (relative to $N$) estimate of the density 
of $(X_t^N)^{-1}_\#(\rho dx)$, it is straightforward (see \cite{FangLuoTh}) to prove that
\begin{equation*}
\lim_{N\ra +\infty}\E\Bigl( \xi \int_\T f_t^N\bigl((X_t^N)^{-1}\bigr)\, dx\Bigr)
=\E\Bigl( \xi \int_\T f_t\bigl((X_t)^{-1}\bigr)\, dx\Bigr).
\end{equation*}
But by Proposition \ref{prop4.4}, $\dis \int_{\T} f_t^N\bigl((X_t^N)^{-1}\bigr)\, dx=\int_\T g_t^N(x)\, dx=0$.
Then for any bounded $\xi$, 
\begin{equation*}
\E\Bigl( \xi \int_\T f_t\bigl((X_t)^{-1}\bigr)\, dx\Bigr)=0.
\end{equation*}
The result \eqref{eq4.31} follows. 
\end{proof}

\section{Appendice}\label{sect5}

First we recall the notations: $\dis \phi_{2k-1}(x)=\frac{\sin(kx)}{k},\ \phi_{2k}(x)=-\frac{\cos(kx)}{k}$ and $\alpha_k=k^q$
for some $q>1$. Let $(X_t^N)$ and $(X_t)$ be solutions to SDE: 

\begin{equation}\label{eq5.1}
dX_t^N=\sum_{k=1}^N \frac{1}{\alpha_k} \Bigl( \partial_x\phi_{2k-1}(X_t^N)\, dB_{2k-1}(t)
+   \partial_x\phi_{2k}(X_t^N)\, dB_{2k}(t)\Bigr),
\end{equation}
\begin{equation}\label{eq5.2}
dX_t=\sum_{k=1}^\infty \frac{1}{\alpha_k} \Bigl( \partial_x\phi_{2k-1}(X_t)\, dB_{2k-1}(t)
+   \partial_x\phi_{2k}(X_t)\, dB_{2k}(t)\Bigr)
\end{equation}
respectively. 

\begin{proposition}\label{prop5.1} Let $p\geq 1$ be an integer, there exists a constant $C_p>0$ independent of $x$ such that
\begin{equation}\label{eq5.3}
\E\Bigl[ \Bigl( X_t^N(x)-X_t(x)\Bigr)^{2p}\Bigr]
\leq \frac{C_p}{N^{2q-1}}.
\end{equation}
\end{proposition}

\begin{proof}
Let $\eta_t= X_t(x)-X_t^N(x)$.  Then 
\begin{equation*}
\begin{split}
d\eta_t=& \sum_{k=1}^N \frac{1}{\alpha_k} \Bigl[ \bigl( \cos(kX_t)-\cos(kX_t^N)\bigr)\,dB_{2k-1}(t)+
\bigl( \sin(kX_t)-\sin(kX_t^N)\bigr)\,dB_{2k}(t)\Bigr]\\
&+  \sum_{k>N} \frac{1}{\alpha_k} \Bigl[  \cos(kX_t)\,dB_{2k-1}(t)+
\sin(kX_t)\,dB_{2k}(t)\Bigr].
\end{split}
\end{equation*}
Then It\^o stochastic contraction $d\eta_t\cdot d\eta_t$ admits the expression
\begin{equation*}
\sum_{k=1}^N \frac{1}{\alpha_k^2} \Bigl[ \bigl( \cos(kX_t)-\cos(kX_t^N)\bigr)^2
+ \bigl( \sin(kX_t)-\sin(kX_t^N)\bigr)^2\Bigr]+ \sum_{k>N} \frac{1}{\alpha_k^2}, 
\end{equation*}
which is equal to 
\begin{equation*}
\sum_{k=1}^N \frac{4}{k^{2q}}\, \sin^2\Bigl(k\, \frac{X_t-X_t^N}{2}\Bigr) + \sum_{k>N}\frac{1}{k^{2q}}; 
\end{equation*}
the first above sum is dominated by 
\begin{equation*}
\sum_{k=1}^{+\infty} \frac{1}{k^{2q-2}}(X_t-X_t^N)^2,
\end{equation*}
while the second sum has the upper bound $\dis \frac{1}{N^{2q-1}}$. Hence there is a constant $C>0$ independent of $x$ such that 
\begin{equation}\label{eq5.4}
d\eta_t\cdot d\eta_t \leq \Bigl(C\, \eta_t^2 + \frac{1}{N^{2q-1}}\bigr)\, dt.
\end{equation}
By It\^o\textcolor{black}{'s} formula, $\dis d\eta_t^{2p}=2p\eta_t^{2p-1}d\eta_t + 2p(2p-1)\eta_t^{2p-2} \, d\eta_t\cdot d\eta_t$, which is dominated, 
according to \eqref{eq5.4}, by 
\begin{equation*}
2p\eta_t^{2p-1}d\eta_t + 2p(2p-1)C\, \eta_t^{2p}\, dt + 2p(2p-1)\frac{\eta_t^{2p-2}}{N^{2q-1}}dt.
\end{equation*}
Note that $\dis 2p(2p-1)\eta_t^{2p-2}\leq 2p(2p-1)(2\pi)^{2p-2}$ that is denoted by $C_p'$. Therefore 
\begin{equation*}
d\eta_t^{2p}\leq 2p\eta_t^{2p-1}\, d\eta_t + C_p'\, \eta_t^{2p} dt + \frac{C_p'}{N^{2q-1}} dt.
\end{equation*}

It follows that 
\begin{equation*}
\E(\eta_t^{2p})\leq C_p'\, \int_0^T \E(\eta_s^{2p})\, ds + \frac{C_p' t}{N^{2q-1}}.
\end{equation*}
Then Gronwall\textcolor{black}{'s} lemma yields 
\begin{equation*}
\E(\eta_t^{2p})\leq \frac{C_p' t}{N^{2q-1}}\, e^{C_p' t},
\end{equation*}
that is nothing but \eqref{eq5.3}.  
\end{proof}

Now we denote by $\dis \tilde K_t^N$ the density of $(X_t^N)^{-1}_\#(dx)$ with respect to $dx$ and 
$\dis \tilde K_t$ the density of  $(X_t^{-1})_\#(dx)$ with respect to $dx$. By Kunita formula \cite{Kunita}, 
we have
\begin{equation*}
\tilde K_t^N = \exp\Bigl( \sum_{k=1}^{2N}\frac{1}{\alpha_k}\int_0^t \partial_x^2\phi_k(X_s^N)\circ dB_s^k\Bigr),
\end{equation*}
and
\begin{equation*}
\tilde K_t = \exp\Bigl( \sum_{k=1}^{+\infty}\frac{1}{\alpha_k}\int_0^t \partial_x^2\phi_k(X_s)\circ dB_s^k\Bigr).
\end{equation*}
First we consider 
\begin{equation*}
\hat K_t^N = \exp\Bigl( \sum_{k=1}^{2N}\frac{1}{\alpha_k}\int_0^t \partial_x^2\phi_k(X_s^N)\, dB_s^k\Bigr),
\end{equation*}
and
\begin{equation*}
\hat K_t = \exp\Bigl( \sum_{k=1}^{+\infty}\frac{1}{\alpha_k}\int_0^t \partial_x^2\phi_k(X_s)\, dB_s^k\Bigr).
\end{equation*}
Then they are linked by
\begin{equation}\label{L}
\tilde K_t^N=\hat K_t^N\,exp\Bigl(-\sum_{k=1}^{2N} \frac{k^2}{k^{2q}}\Bigr).
\end{equation}

To estimate $ \hat K_t^N$ in $L^p$, we use exponential martingale. More precisely, 
for a continuous martingale $(M_t)$ with quadratic variation $<M>_t$, we write down, for $p\geq 1$,
\begin{equation*}
e^{pM_t}=e^{pM_t-p^2<M>_t}\, e^{p^2 <M>_t}.
\end{equation*}
By Cauchy-Schwarz inequality and the fact $\dis \E\bigl( e^{2pM_t-2p^2<M>_t}\bigr)=1$, we get 
\begin{equation}\label{eq5.5}
\E(e^{pM_t})\leq \Bigl( E\bigl( e^{2p^2<M>_t}\bigr)\Bigr)^{1/2}.
\end{equation}
For $\dis M_t=\sum_{k=1}^{2N}\frac{1}{\alpha_k}\int_0^t \partial_x^2\phi_k(X_s^N)\, dB_s^k$, we have
\begin{equation*}
<M>_t =\sum_{k=1}^N \int_0^t \frac{1}{\alpha_k^2} \bigl( \partial_x^2\phi_k(X_s^N)\bigr)^2\, ds
\leq \sum_{k=1}^{+\infty} \frac{t}{k^{2q-2}}=\xi(2q-2)\, t.
\end{equation*}
 Then by \eqref{eq5.5}, we get
\begin{equation}\label{eq5.6}
\E\Bigl( (\hat K_t^N)^p\Bigr) 
\leq e^{p^2\xi(2q-2)\, t}.
\end{equation}
Obviously above estimate holds true for $\tilde K_t$, that is, 
\begin{equation}\label{eq5.7}
\E\bigl( \hat K_t^p\bigr) 
\leq e^{p^2\xi(2q-2)\, t}.
\end{equation}

\begin{proposition}\label{prop5.2}
For $\dis q>\frac{5}{2}$, choose $\delta>0$ such that $q-\frac{1}{2}<q-\delta<2q-3$; then
there is a constant $C_p>0$ such that 
\begin{equation}\label{eq5.8}
\E\Bigl( (\tilde K_t^N-\tilde K_t)^{2p}\Bigr)\leq \frac{C}{N^{q-\delta}}.
\end{equation}
\end{proposition}

\begin{proof}
Let $\dis \zeta_t =\hat K_t^N-\hat K_t$. we have 
\begin{equation*}
d\hat K_t^N =\hat K_t^N\, \sum_{k=1}^{2N}\frac{1}{\alpha_k}\partial_x^2\phi_k(X_t^N)\, dB_t^k, 
\end{equation*}
and 
\begin{equation*}
d\hat K_t=\hat K_t\, \sum_{k=1}^{+\infty}\frac{1}{\alpha_k}\partial_x^2\phi_k(X_t)\, dB_t^k.
\end{equation*}
So 
\begin{equation*}
\begin{split}
d\zeta_t=&\sum_{k=1}^{2N}\frac{1}{\alpha_k}\Bigl( \hat K_t^N\,\partial_x^2\phi_k(X_t^N)
-\hat K_t\,\partial_x^2\phi_k(X_t)\Bigr)\, dB_t^k\\
&-\sum_{k>2N} \frac{1}{\alpha_k}\, \textcolor{black}{\hat K}_t\, \partial_x^2\phi_k(X_t)\, dB_t^k.
\end{split}
\end{equation*}
Now the It\^o stochastic contraction $d\zeta_t\cdot d\zeta_t$ admits the expression 

\begin{equation*}
\begin{split}
d\zeta_t\cdot \textcolor{black}{d}\zeta_t=&\sum_{k=1}^{2N}\frac{1}{\alpha_k^2}\Bigl( \hat K_t^N\,\partial_x^2\phi_k(X_t^N)\,
-\hat K_t\,\partial_x^2\phi_k(X_t)\Bigr)^2\, dt\\
&+\sum_{k>2N} \frac{1}{\alpha_k^2}\, \bigl(\textcolor{black}{\hat K}_t\, \partial_x^2\phi_k(X_t)\bigr)^2\, dt.
\end{split}
\end{equation*}
Remark that 

\begin{equation*}
\hat K_t^N\,\partial_x^2\phi_k(X_t^N)\,
-\hat K_t\,\partial_x^2\phi_k(X_t)
= (\hat K_t^N-\hat K_t)\, \partial_x^2\phi_k(X_s^N)
+ \hat K_t\, \bigl(\partial_x^2\phi_k(X_s^N)-\partial_x^2\phi_k(X_s)\bigr),
\end{equation*}
then
\begin{equation*}
|\hat K_t^N\,\partial_x^2\phi_k(X_t^N)\,
-\hat K_t\,\partial_x^2\phi_k(X_t)|
\leq k\,|\hat K_t^N-\hat K_t|
+ k^2\,\hat K_t\,  |X_s^N-X_s|.
\end{equation*}

Therefore we get the following upper bound
\begin{equation}\label{eq5.9}
d\zeta_t\cdot d\zeta_t
\leq 2\xi(2q-2)\, |\hat K_t^N-\hat K_t|^2+ 2\, \xi(2q-4)\, \hat K_t^2\, |X_t^N-X_t|^2
+\frac{\hat K_t^2}{N^{2q-3}}.
\end{equation}

For $p\geq 1$ an integer, by It\^o\textcolor{black}{'s} formula and above estimate, there are three constants $C_1, C_2, C_3$
only dependent of $p$ such that 
\begin{equation*}
\begin{split}
d\zeta_t^{2p}&=2p\zeta_t^{2p-1}d\zeta_t + 2p(2p-1)\zeta_t^{2p-2}\, d\zeta_t\cdot \textcolor{black}{d}\zeta_t\\
&\leq 2p\zeta_t^{2p-1}d\zeta_t+ C_1(p) \zeta_t^{2p}\,dt + C_2(p) \zeta_t^{2p-2}\hat K_t^2 (X_t^N-X_t)^2\, dt
+C_3(p) \frac{\zeta_t^{2p-2}\hat K_t^2}{N^{2q-3}}.
\end{split}
\end{equation*}
It follows that 
\begin{equation*}
\begin{split}
\E(\zeta^{2p})&\leq C_1(p)\int_0^t \E(\zeta_s^{2p})\, ds + C_2(p)\int_0^t \E\Bigl( \zeta_s^{2p-2}\hat K_s^2 (X_s^N-X_s)^2\Bigr)ds\\
&+\frac{C_3(p)}{N^{2q-3}}\int_0^t \E\Bigl( \zeta_s^{2p-2}\hat K_s^2\Bigr)\, ds.
\end{split}
\end{equation*}
Let $r>1$ and $\hat r>1$ such that $(1/r) + (1/\hat r)=1$. 
Now by \eqref{eq5.5} and \eqref{eq5.6}, we have 
\begin{equation*}
C=\sup_{s\in [0,1]}\E\Bigl( \bigl(\zeta_s^{2p-2}\hat K_s^2\bigr)^{\hat r}\Bigr) <+\infty.
\end{equation*}
On the other hand, by \eqref{eq5.3}, $\dis \E\bigl( (X_s^N-X_s)^{2r}\bigr)\leq \frac{C_r}{N^{2q-1}}$. Taking 
$\dis r=\frac{2q-1}{q-\delta}$ yields 
\begin{equation*}
\E(\textcolor{black}{\zeta_t}^{2p})\leq C_1(p)\int_0^t \E(\zeta_s^{2p})\, ds + \frac{C_2(p)}{N^{q-\delta}}\, t.
\end{equation*}
Gronwall lemma gives 
\begin{equation*}
\E\Bigl( (\hat K_t^N-\hat K_t)^{2p}\Bigr)\leq \frac{C}{N^{q-\delta}}.
\end{equation*}
Combing this with Relation \eqref{L}, we obtain the desired estimate \eqref{eq5.8}. 
\end{proof}

\vskip 4mm
{\bf Acknowledgement:} This work has been taken from a part of the PhD thesis \cite{Ding} of the first named author by a joint PhD program between the
Academy of Mathematics and Systems Science, Chinese Academy of Sciences and 
the Institute of Mathematics of Burgundy, University of Burgundy. 
He is grateful to the hospitality of these two institutions. The financial support from
China Scholarship Council and National Center for Mathematics and Interdisciplinary Sciences are particularly acknowledged. The third author has been partially supported by  National Key R$\&$D Program of China (No. 2020YF0712700) and 
National Natural Science Foundation of China (Grant No. 12171458).

\end{document}